\documentclass[12pt]{article}

\usepackage{xcolor}

 \usepackage[colorlinks=true]{hyperref}
\hypersetup{urlcolor=blue, citecolor=red}

\usepackage{amsmath} 
\usepackage{amsthm}  
\usepackage{amssymb} 
\usepackage{enumitem} 

\usepackage[letterpaper]{geometry}
\newcommand{\comment}[1]{}
\newcommand{\complexes}{{\bf C}}

\newcommand{\reals}{{\bf R}}

\newcommand{\supp}{\mathop{\rm supp}\nolimits }

\newcommand{\BibTeX}{{\rm B\kern-.05em{\sc i\kern-.025em b}\kern-.08em     
    T\kern-.1667em\lower.7ex\hbox{E}\kern-.125emX}}

\newcommand{\re}{\mathop{\rm Re}\nolimits} 
\newcommand{\im}{\mathop{\rm Im}\nolimits} 

\newcommand{\rv}[1]{{\color{blue} #1}}
\renewcommand{\rv}[1]{#1}

\usepackage{hyperref}

\newenvironment{remark}[1][Remark]{\begin{trivlist}\item[\hskip \labelsep
{\it #1. }]}{  \goodbreak \end{trivlist}}  

\numberwithin{equation}{section}

\newtheorem{theorem}[equation]{Theorem}
\newtheorem{proposition}[equation]{Proposition}

\newtheorem{lemma}[equation]{Lemma}

\renewcommand{\thetheorem}{\arabic{section}.\arabic{theorem}}
\renewcommand{\theequation}{\arabic{section}.\arabic{equation}}

\begin{document}

\title{Recovering nonsmooth coefficients for higher-order perturbations of a 
polyharmonic operator} 

\author{R.M.~Brown\footnote{R.~Brown is partially supported by a grant from the Simons
Foundation (\#422756).}
\footnote{Corresponding author, rbrown@uky.edu}
\\ Department of Mathematics\\ University of Kentucky \\ Lexington, Kentucky, USA
\and 
D.~Faraco\footnote{D.Faraco is partially supported by the Severo Ochoa Programme
for Centers of Excellence Grant CEX2019-000904-S funded by MCIN/AEI/10.13039/501100011033, 
by the  ERC Advanced Grant 834728, by  project  PID2021-124195NB-C32. and 
by  CM through the Line of Excellence for University Teaching Staff between CM and UAM}\\
Departamento de Matem\'aticas \\ Universidad Aut\'onoma de Madrid,\\
 Instituto de
Ciencias Matem\'aticas, CSIC-UAM-UC3M-UCM\\
Madrid, Spain
\and
L.D.~Gauthier\\Department of Mathematics \\ Carthage College \\ Kenosha, Wisconsin, USA   }
\date{}
\maketitle


\date{}
\maketitle

\newpage

\begin{abstract}
We consider an inverse problem for a higher order elliptic operator 
where the principal part is the 
polyharmonic operator $(-\Delta)^m$ with $ m \geq 2$. We show 
that the map from the coefficients to a 
certain bilinear form 
is injective. We have a particular focus on obtaining these results under lower
regularity on the coefficients. It is known that knowledge of this 
bilinear form is equivalent to a knowledge of 
a Dirichlet to Neumann map or the Cauchy data for solutions. 
\end{abstract}

\section{Introduction}\label{Intro}

In this paper, we consider inverse boundary value problems for 
 operators with principal part the $2m$th order polyharmonic operator 
 with $m\geq 2$  on a 
 domain $ \Omega \subset \reals ^d$ with $d\geq 3$: 
\begin{equation}
\label{e:IntroOp}
 L=(-\Delta ) ^ m + \sum _ { k=0}^m \sum _ { |\alpha |= k} A^{(k)}_ \alpha D^ \alpha .
 \end{equation}
Without loss of generality, we assume that the array  of coefficients 
$A^{(k)}$ is symmetric in the sense that the value of a coefficient $ A_\alpha^{(k)}$ 
is not changed when we permute the indices  in $ \alpha$.  We will see that this
array may be viewed as a symmetric tensor.

Our main result is the following theorem giving conditions which guarantee 
that the coefficients in an operator \eqref{e:IntroOp} are determined by boundary 
information.   We refer the reader to later sections for explanations 
of the notation that appear in this theorem. 

\begin{theorem}
    \label{t:Main}
    Let $L_1$  and $L_2$ be as in \eqref{e:IntroOp} with coefficients 
    $A^{(k)}_ \ell =(  A^{(k)} _ {\ell,\alpha} )_ { |\alpha| =k}$, 
    $\ell = 1,2$, $k=0,\dots, k_0$ and that the coefficients $A^{(k)}_\ell= 0$ 
    for $ k > k_0$. 
    We give two cases for parameters that we will use in our 
    hypotheses on the coefficients: 
    \begin{enumerate}[label=(\alph*)]
        \item $k_0= \lfloor m/2 \rfloor -1$ and 
    $m/2 + 1/2 <  s< m/2+1$,   
    \item $m$ odd, $k_0= m/2 -1/2$, $ s= m/2 + 1/2$,
    \end{enumerate}
    We assume that the coefficients 
    $A_\ell^{(k)} \in \tilde W ^ { k-s, p} (\Omega)$ for $k=0,\dots, k_0$, 
    for $s$ as above and $p$ satisfying $p\geq 2$ and $ 1/p< (m-s)/d$.   
    Assume that the bilinear forms for the operators are equal 
    as defined in \eqref{e:form}. Then the coefficients 
    $ A_1^ {(k)} = A_2^{(k)}$ for $k =0, \dots , k_0$. 
\end{theorem}

To help put these results in perspective, we compare our work with earlier results in a few specific cases which
highlight the cases where our results improve on earlier work. 

\rv{
\begin{center}
   \begin{tabular}{cccc}
              &Assylbekov, Iyer \cite{MR4027047}& Brown, Gauthier \cite{MR4455267} & Theorem \ref{t:Main} \\ 
    $m=4$ &   $k_0=1$, $s<2$  & $k_0 =1$, $s<3$    & $k_0 =1$, $s<3$   \\
    $m=5$ & $k_0 =1$, $s<5/2$ & $k_0 =1$, $s< 7/2$ & $k_0 =1$, $s<7/2$ \\
 $m=5$    &                   &                    & $k_0 =2$, $s=3$   \\
    $m=6$ &  $k_0=1$, $s<3$   & $k_0=1$, $s<4$     &  $k_0= 2$, $s<4$
    \end{tabular}
\end{center}
}

The study of inverse problems for the biharmonic operator was 
initiated by Krupchyk, Lassas and Uhlmann 
\cite{MR2873860}
 who consider operators with terms involving derivatives up to first 
 order   and establish that 
the coefficients of the lower order terms are uniquely determined 
by measurements of  the Dirichlet to Neumann 
map for the operator on certain subsets of the boundary.  This work requires that the 
zeroth order term be bounded and measurable 
and the coefficients in the first order term must be Lipschitz.  
An extension of this result to 
polyharmonic operators appeared 
in a subsequent paper also by Krupchyk, Lassas and Uhlmann \cite{MR3118392}. 
 We  will limit our summary 
to papers focusing on 
the regularity needed to obtain a uniqueness result and 
refer the reader to the papers cited for additional background. 
The work of Assylbekov  \cite{MR3627033,MR3691848} 
considers operators where  the zeroth order 
term $A^{(0)}$ has $ s>-m/2$ derivatives in a certain 
$L^p$-space and the first order term 
is in a Sobolev space of order $ -m/2 + 1 $. Under these assumptions and with a 
restriction on the dimension, they are able to establish that the 
coefficients are determined by the Dirichlet to Neumann map for 
the operator. The restriction on the dimension was removed by 
Assylbekov and Iyer in a later paper \cite{MR4027047}. A 
related, interesting  result is due to Krupchyk and 
Uhlmann \cite{MR3484382} who consider operators with zeroth order 
term  $A^{(0)}$ in the $L^p$ space with 
$p>n/2m$.  While these potentials are 
smoother than those considered by 
Assylbekov and Iyer, it is interesting to note that 
the result of Krupchyk and Uhlmann does not 
follow from the work of Assylbekov and Iyer by Sobolev embedding. 

There has been much more progress in the study of 
regularity hypotheses for inverse problems involving 
operators where the principal part is the Laplacian. The 
fundamental work here is due to Sylvester 
and Uhlmann from 1987 \cite{SU:1987} who established a uniqueness 
result for the Calder\'on problem 
provided the coefficients in the operator are sufficiently smooth. 
A paper of Brown raises the question of reducing the 
regularity of the coefficient \cite{RB:1994a} and establishes 
a uniqueness result  for the Calder\'on problem which amounts 
to studying certain  operators of the form 
$ \Delta +q$ where the potential $q$ is a distribution of smoothness order $ -s$
with $-s>-1/2$.  An important advance 
was made by Haberman and Tataru who were able to treat 
operators of the form $ \Delta +q$ when the potential 
$q$ is a derivative of a uniformly continuous, compactly supported function 
\cite{MR3024091}.  The key 
innovation of their work is to introduce spaces which are defined 
using an operator that arises in 
the problem. These spaces are (essentially) the $X^s_{h \zeta}$ 
spaces which are defined later 
and play a key role in this paper. Rogers and Caro \cite{MR3456182} 
extends Haberman and 
Tataru's result to include operators where $q$ is the derivative of an 
$L^\infty$-function. More recently, Haberman \cite{MR3397029}  
established a uniqueness result on $ \reals ^d$ which allows 
the $q$ to be the derivative of an $L^d$ function when $d=3$ or $4$. 
This verifies a conjecture of Uhlmann. There are further 
results by  Ham, Kwon, and Lee \cite{MR4273826} and 
Ponce-Vanegas \cite{MR4112132} but the conjecture in higher 
dimensions remains open. 

\rv{
We believe that the study
of inverse boundary value problems for polyharmonic operators 
 is of interest for several reasons.
Polyharmonic operators with higher-order perturbations arise 
naturally in higher-order elasticity models, such as Mindlin’s strain-gradient 
elasticity \cite{MR160356} and its subsequent developments by  Aifantis \cite{Aifantis1992}
and Reiher and Bertram \cite{MR3873605}
The monograph of Gazzola, Grunau and Sweers \cite{MR2667016} gives applications of polyharmonic
operators in subjects ranging from elasticity to conformal geometry.
There is also interest in non-smooth coefficents for several reasons. 
 Allowing non-smooth coefficients
is of interest because not all physical models will give rise to equations with non-smooth coefficients. 
In addition, it is of interest to explore the limits of the techniques we have developed
to study these problems. This will either help us to understand the fundamental limitations in 
these problems or point us to areas where we can improve our tools for studying inverse boundary value
problems. 
}

One important inspiration for our work is several papers studying
inverse boundary value problems for 
operators where the principal part is the polyharmonic operator and
there are terms of order greater than one.  The first work here is due to 
Bhattacharyya and Ghosh 
\cite{MR3953481} who establish that for an operator with 
terms up to order $2$, the coefficients
are uniquely determined by a Dirichlet to Neumann map, at least 
when $m\geq 3$. A partial 
result is available when $m=2$.
\rv{
There has been much additional work on inverse problems for polyharmonic operators
that we will not attempt to summarize  here. We direct readers
to recent interesting work by 
 Bhattacharyya, Krishnan, and Sahoo \cite{MR4631015} 
that
establishes uniqueness for operators 
with terms of order up to $m-1$. This  work also considers a  
variant where we are only given 
information about the Dirichlet to Neumann map on part of the boundary. 
This has been extended to certain terms up to order $m$ by
Bhattacharyya and Kumar \cite{MR4723844}. 
These interesting papers do not consider non-smooth coefficients 
but they provide a good summary of recent progress for other aspects
of inverse boundary problems for polyharmonic operators. 
}

The research reported here builds most directly on a recent paper 
by two of us, Brown and Gauthier \cite{MR4455267}. 
These authors adapt the argument of Haberman and Tataru to the study 
of polyharmonic operators with 
lower order terms up to first order. For an operator 
with only a zeroth order term, we consider 
operators with coefficients that 
lie in a Sobolev space of order $s$ with $ s> -2$. This 
smoothness index may be close to optimal when the principal 
part is the biharmonic operator. In section we \ref{s:further} conjecture 
that the optimal result when the principal part is $ (-\Delta)^m$ will  be  $ s=-m$. 
We also give results for an operator with a first order term, but 
unfortunately the result as stated in our previous work 
\cite[Theorem 1.3]{MR4455267} is not right.  In our 
statement of that result, we neglected to include a condition that is needed 
in the proof of one of the estimates. As a result, when there is a first 
order term our result requires that $m \geq 4$.  With this additional 
assumption the proof given in our previous paper is correct and this
corrected result is contained in case a) of our main result below, 
Theorem \ref{t:Main} and a recently published correction \cite{MR4455267}. 
The additional observation is that we need the 
condition $s-k-2\sigma \geq 0$ which appears in \eqref{SupCond} 
of Proposition \ref{p:Avg} below
(see also our 
previous work \cite[Theorem 4.6]{MR4455267}).   

\rv{
The key advance in this work is to begin the study of inverse problems for
polyharmonic operators with non-smooth coefficients 
as in \eqref{e:IntroOp}  with terms involving derivatives of order greater than one.}
As part of our argument, we establish that the coefficients satisfy a series
of equations and develop arguments involving symmetric tensors to study these equations.
\rv{
There are significant overlaps in our method with the arguments using the momentum ray transform
in \cite{MR4631015,MR4723844,SS:2024}. Our work depends on solving a linearized Calderón problem
for a polyharmonic problem where a limited set of polyharmonic functions can be used. 
In a recent, 
interesting work, Sahoo and Salo  \cite{SS:2024} give  
an almost complete solution of the linearized Calder\'on problem 
for the polyharmonic operator by studying momentum ray transforms. 
The work by Bhattacharya and collaborators \cite{MR4631015,MR4723844}
mentioned above
also uses the momentum ray transform to study inverse boundary value problems
for operators with lower order terms up to order $m$. 
We choose to present our argument for several reasons. Some may find it helpful
to have a complete presentation without having to translate notation back and forth between
several papers. More importantly, in our argument the limited regularity of the coefficients
imposes restrictions on the expressions we can study. By giving the full argument, it is easier 
to make sure that we follow these restrictions. 
} 

The outline of the paper is as follows. In section \ref{s:notation}, 
we introduce notation
\rv{for tensors, definitions of function spaces we need 
and summarize some of the important properties of these spaces that we will use. 
Section \ref{s:CGO} describes the construction of the CGO solutions we use and shows how
the averaging method of Haberman and Tataru can be used to improve the asymptotic estimates 
for a sequence of these CGO solutions. This improvement is the key to our efforts to study
operators with less regular coefficients. 
In Section \ref{s:maineqn} we substitute our solutions
into a bilinear form related to the operators and obtain a set of
equations involving the Fourier transforms of the coefficients. 
Finally, in section \ref{s:uniq}
we show how these equations uniquely determine the coefficients.  This section
is largely algebraic in nature. 
The   subsection \ref{sec:tensorstructure} gives a result on the structure of 
a family of tensors that arise in our argument.  
Finally, Appendix \ref{s:tensorcomp} 
includes several  algebraic and combinatorial results related to symmetric 
tensors that are used in the main part of the paper. 
}

In most of our arguments we will suppress constants 
in inequalities by using the notation
$a\lesssim b$ to mean that there is a constant $C$ so that 
$ a \leq C b$. It will be an important
point that the constant is independent of the semi-classical 
parameter $h$, at least for $h$ small, as this will be 
needed when we take limits as $h $ tends to zero.

\section{Notation}
\label{s:notation}
We use $ D_k = \displaystyle\frac 1 i \frac \partial { \partial x_k}$ 
to denote a partial derivative and  
$ \alpha = (\alpha _1, \alpha _2, \dots , \alpha _j ) \in \{ 1, \dots, d\} ^ j $
to denote a multiindex. We let $ D= (D_1, \dots, D_d)$ 
and $ \nabla = iD$.  We write higher-order derivatives as  
$ D^ \alpha = D_{ \alpha _1} \dots D_{ \alpha _j}$. We use 
$|\alpha|=j$ to denote the number of elements in $ \alpha$.  The multiindices 
we define here will be useful when we formulate this equation using symmetric tensors, 
but they differ from the multiindices that usually appear in the 
pde literature. At one point 
in section \ref{sec:tensorstructure} we will need to use both types of multiindices.

Since the derivative $D^\alpha$ does not depend on the order 
that we take the derivatives, 
it is natural to assume that the coefficients 
$A_\alpha$ are unchanged by a permutation of the entries in $\alpha$. 
This leads us to view the coefficient as a symmetric tensor. 
For our purposes,  we 
define $A$,  a tensor of order  $j$,   as an array of 
complex numbers $A_\alpha$ 
indexed by multiindices $\alpha$ with $|\alpha|=j$. 
We let  $\Pi(j)$ be the group of permutations of $ \{ 1,\dots , j\}$.  
For $ \pi \in \Pi (j)$ and
$ \alpha$ a multiindex of length $j$,  we define $ \pi (\alpha) = 
( \alpha _{ \pi (1)}, \alpha _ { \pi (2) }, \dots \alpha _{ \pi (j)} )$. 
In this paper, our tensors will always be symmetric, which means that 
$ A_ \alpha$ is unchanged by permutation of the indices, that is, we have
$ A_ \alpha = A_{ \pi ( \alpha)}$ for all $ \pi \in \Pi(j)$ 
and all multiindices $\alpha$ of length $j$. 
We let 
$ S^j ( \complexes ^d)$ denote the collection of $j$th 
order symmetric tensors for $\complexes ^d$. 
If 
$ A = (A_ \alpha ) _{ | \alpha| = j }$ is an array  
indexed by $ \alpha \in \{ 1, \dots, d \} ^j$, then 
we define the symmetrization map $ \sigma$ by 
\[
\sigma(A)_\alpha = \frac 1 { j !} \sum _{ \pi \in \Pi(j)} A_ { \pi (\alpha)}.
\]

If $ v^j= ( v^j_1, \dots, v^j _d) \in \complexes^d$,
$ j =1 ,\dots, k$, we can define an element of 
$ S^k( \complexes ^d)$ by 
$$ v^1 \otimes v^2 \otimes \dots \otimes v^k 
= \sigma ( ( v^1_{ \alpha _1}v^2_{ \alpha _2 }
\dots v^k _{ \alpha _k} )_ { |\alpha | = k }). 
$$
More generally, if $ A\in S^j(\complexes^d)$ and 
$ B \in S^k ( \complexes ^d)$, we define the symmetric tensor product
$ A\otimes B \in S^ { j+k}( \complexes ^d)$ by 
\begin{equation} \label{eq:deftensorproduct}
(A\otimes B)_ \alpha = \sigma ( A_ {(\alpha_1, \dots , \alpha_j)} 
B _ { (\alpha _ {j+1},\dots ,\alpha _ { j+k} ) })\quad |\alpha|=j+k.
\end{equation}
For a tensor $A$, we will use $A^k$ to denote the $k$-fold tensor product, 
$ A^k = A\otimes A^ {k-1}$ and $ A^1 = A$. An important special case of this 
notation that we will use frequently is to 
write $ v^k$ for the $k$-tensor which is the $k$th symmetric tensor 
power of a vector $ v \in \complexes ^d$. 
We will often write $ A^{(k)}$ 
to stand for a tensor of order $k$. The parentheses serve to distinguish 
a tensor of order $k$ from a $k$th tensor power. 

As a generalization of the inner product, given  tensors $A$ of order $k$ and  
$B$ of order $j$ with $k\geq j$, 
we may define the contraction of $A$ by $B$ as the tensor 
$A\cdot B\in S^{ k-j} ( \complexes ^d)$ whose components are given by 
\begin{equation}
\label{e:ConDef}
(A\cdot B)_\alpha = \sum_{ |\beta |= j} A_{\alpha\beta} B_\beta. 
\end{equation}
Here, $ \alpha$ is a multiindex of length $k-j$,  $\beta$ is of length $j$ and
$\alpha\beta$ is the multiindex of length $k$ formed by concatenation. 
Note that this definition continues to make sense if the arrays $A$ 
and $B$ are not symmetric. At a few points in the arguments below, we will 
make use of this definition for nonsymmetric arrays. 
If the array  $A$ is symmetric, the tensor defined by 
$A\cdot B$ will be symmetric whether or not $B$ is symmetric. 
In the case when $ j=k$, then $(A,B) \to A \cdot B$ gives a bilinear, 
symmetric map from $ S^k ( \complexes^d) \times
S^k ( \complexes ^d) \to \complexes $. We will call this 
map the dot product of $A$ and $B$.

Finally,  we notice that we may identify the space  
$ S^k ( \complexes ^d)$ with the space of 
symmetric complex multilinear maps 
$ A : (\complexes ^d)^k\to \complexes$. If $A \in S^k ( \complexes ^d)$, 
then we define 
a symmetric multilinear map by defining 
\[
A[e_{\alpha_1},e_{\alpha_2},\dots, e_{\alpha _k}]=A_{\alpha_1,\dots, \alpha_k}. 
\]
Here, $ e_ {\alpha_j}$ is a standard basis vector for $ \complexes ^d$. 
Since the map  $A[\, \cdot\,]$ is to be multilinear, it 
is determined by its action on basis elements.
And if we have a symmetric multilinear map, the above 
relation serves to define a symmetric tensor.

We let $ X( \Omega, S^j(\complexes ^d))$ denote the 
functions or distributions 
in the space $X$ which take values in  $S^j ( \complexes ^d)$. 
More concretely, an 
element of 
$ X(\Omega, S^j ( \complexes ^d))$ is a symmetric array 
$ ( A_{ \alpha} ) _{ | \alpha | =  j}$ of distributions in the space $X$. 
If $u$ is a complex valued distribution then $D^j u = ( D^ \alpha u ) _ { |\alpha | = j }$ 
gives a 
function or distribution taking values in $ S^j ( \complexes ^d) $. 

With our notation for tensors, we can write our operator $L$ from \eqref{e:IntroOp} as 
\begin{equation}\label{e:OpDef}
Lu = ( -\Delta )^m u 
+ \sum _{ k =0 } ^ m A^{( k)} \cdot  D^ k u . 
\end{equation}
Here, $ A^{(k)}$ is a distribution taking values in 
$S^k( \complexes ^d)$ which gives the coefficients 
for the terms of order $k$. 
\rv{
To understand why we only consider symmetric coefficients, suppose that we have an array $A^{(k)}$ indexed by 
multi-indices of length $k$. Since $D^k u$ is symmetric, Lemma \ref{l:ConLem} gives that 
the expressions  $A^{(k)} \cdot D^ku$ and $\sigma(A^{(k)})\cdot D^ku$
are equal. Thus, allowing non-symmetric coefficients does not enlarge the class of 
operators being studied, but allows us to write an operator in many different ways. Since the operator and hence the family of
solutions only depend on the
symmetrization of the coefficients, we can only hope to recover the symmetrization of the coefficients from 
information derived from the family of solutions. 
}

Throughout this paper, we will assume that the domain $ \Omega \subset \reals ^d$ is a 
bounded, connected open set with 
smooth  boundary.
Due to the lack of smoothness of the coefficients, we will need to consider 
weak solutions of the equation $ Lu =f$ and carefully define the 
products $A^{(k)}\cdot D^ku$ appearing in the equation. For $ s \in\reals$ 
and $ 1<p< \infty$, we will use $ W^ { s,p} ( \reals ^d)$ to denote the Sobolev 
space of distributions for which the norm 
$ \| ( 1-\Delta ) ^ { s/2} u\|_{ L^p}$ is finite. For an open set $ \Omega$,  
$W^ { s,p} ( \Omega) $ will denote the restriction 
to $ \Omega$ of functions in  $ W^ { s,p } ( \reals ^d)$. 
\rv{
For $s$ a non-negative integer and $ 1< p < \infty$, the spaces $W^{s,p}( \reals ^d)$ coincide with the standard 
Sobolev spaces 
whose derivatives up to order $s$ lie in $L^p( \reals ^d)$ and the spaces for $s>0$ 
may be obtained by complex interpolation from the spaces with $s$ an integer.   
As long as the boundary of $\Omega$ is reasonable, these statements 
continue to hold with $ \reals ^d$ replaced by a domain $\Omega$.
}

The space 
$ W^ { s,p} _0( \Omega)$ 
will denote the closure of $C_c^ \infty( \Omega) $ in $W^ { s,p}( \Omega)$.  
Finally, we will define 
$$
\tilde W ^ { s,p} ( \Omega) = \{ u \in W^ {s,p}( \reals ^d) : 
\supp u  \subset \bar \Omega\}. 
$$
In this definition, $ \supp u$ denotes the support of $u$ 
as a distribution. Note that the notation 
$ \tilde W^ { s,p} ( \Omega)$ is misleading since 
the elements of this space may not be distributions on $ \Omega$.   
When the domain $ \Omega$ has 
a Lipschitz boundary, we have that $ \tilde W^ { -s,2}( \Omega) $ is the dual of the 
space $W^ { s,2} ( \Omega) $ (see, for example, the monograph of 
Mclean \cite[Theorem 3.30]{MR1742312}.) 
We will use $\langle \cdot , \cdot \rangle_{\Omega}$.
for the 
bilinear pairing between a Sobolev space on $ \Omega$ and its dual. 
Thus,  for example, when discussing weak solutions of the operator
\eqref{e:OpDef}, we will want to consider the   bilinear map from $
 W^ {-m,2}( \Omega)\times W^ { m,2}_0( \Omega) $ or 
 $\tilde W^ { -m,2}( \Omega) \times W^ { m,2}( \Omega)$.   

Using well-known results about products in Sobolev spaces, we have the estimate  
\cite[equation (2.3)]{MR4455267}  
\begin{equation}\label{e:trilinearSobolev}
    |\langle f D^ \alpha u , v \rangle _\Omega| \lesssim 
    \|f \| _{ \tilde W ^ { |\alpha |-m , t' }( \Omega)} 
    \| u \| _{ W^ { m,2 } ( \Omega) }\| v \| _ { W^ { m,2 } ( \Omega) }
\end{equation}
where $|\alpha | \leq m$ and $t'$ is the H\"older dual of  $t$ which is defined by 
\begin{equation} \label{e:expdef}
\begin{cases}
    \frac 1 t = 1 - \frac m d , \qquad & \frac m d <  \frac 1 2 \\
    \frac 1 t = \frac 1 2 , & \frac m d > \frac 12 \\
    \frac 1 t > \frac 1 2  , & \frac   m d = \frac 1 2 . 
\end{cases}
\end{equation}
The key step of the proof is the observation that if 
$u, v\in W^{m,2}(\Omega)$, then  the product $uv \in W^{m,t}(\Omega)$
(see \cite[\S 4.4.4, Theorems 1,2]{MR98a:47071}, for example).  
Our initial assumption is that the coefficients in the operator $L$ satisfy 
\begin{equation}
\label{e:coeffcond0}
    A^{(k)} \in \tilde W^ {k-m ,t'}( \Omega), \qquad k =0, \dots, m .
\end{equation}
The condition \eqref{e:coeffcond0} is needed to make sense of weak 
solutions, but we will impose 
additional hypotheses on the coefficients in 
\eqref{e:CoeffCond} and \eqref{e:CoeffSobCond} to 
study the inverse problem. In our proof of 
Theorem \ref{t:Main}, we will show that  \eqref{e:CoeffCond} and \eqref{e:CoeffSobCond}
follow from the hypotheses of this theorem. 

We also assume we have a bilinear form 
$ B_0 : W^ { m,2} ( \Omega) 
\times W^ { m,2} ( \Omega) \to \complexes$  given by 
$$
B_0 [ u,v] 
= \int _ \Omega \sum _ {| \alpha| = |\beta| = m} a_{ \alpha\beta } 
D^ \alpha u D^ \beta v \,dx
$$
which satisfies
$$
B_0 [ u,v ] = \int _ { \Omega} [(-\Delta ) ^ m u ]v\,dx ,\quad  u \in 
C^ \infty ( \Omega) , \  v \in C^ \infty _0 ( \Omega) . 
$$
For an operator $L$ as in \eqref{e:OpDef}, we define the bilinear form  
$B_L: W^ { m,2} ( \Omega) \times W^ { m,2} ( \Omega) \to \complexes$ 
by 
\begin{equation}
    \label{e:formdef} 
     B_L[u,v] = B_0[u,v] + \sum _ { k = 0}^m\langle A^{(k)} 
     \cdot D^k u, v\rangle  _\Omega
\end{equation}
where the pairings are well defined thanks to \eqref{e:trilinearSobolev}.
The form is defined for $u$ and $v$ in $W^ { m,2}( \Omega)$,
but in the definitions of weak solutions below, we will 
restrict the form to subspaces of $W^ { m,2} ( \Omega) $. 

Let $f$ be in $ \tilde W^{ -m,2}( \Omega)$. We say that $u$ is a weak solution of 
$Lu=f$ if $u \in W^ { m,2} ( \Omega) $ and we have
$$
B_L[u,v] = \langle f, v\rangle_\Omega ,\qquad v \in W^ { m,2 }_0 ( \Omega) . 
$$
We will also consider local weak solutions where $u$ is 
only required to be in the space 
$W^ { m,2} _{loc} ( \Omega) $ and the test
functions $v$ are required to be in $W^{m,2}_c( \Omega) $, the space of functions in 
$W^ { m,2} ( \Omega) $ which are compactly supported in $ \Omega$.  
We say that $u$ is a solution of the transposed equation, 
$L^t u=f$, if  $u \in W^ {m,2 } ( \Omega) $ and we have 
$$
B_L[v,u]= B_{L^t}[u,v] = \langle f, v\rangle _\Omega, \qquad v \in W_0^ { m,2} ( \Omega) . 
$$

Our goal is to construct complex geometrical optics (CGO) solutions of the 
equation $ Lu = 0$ in a neighborhood of the closure of $\Omega$. 
These will be solutions of the form  
\begin{equation}
    \label{e:cgo}
 u (x) = \exp( i x\cdot \zeta /h) ( a(x) + \psi(x)) .
 \end{equation}
The parameter $h>0$ will be small and the vector $ \zeta $  
is chosen so that the exponential factor is a solution of 
$ \Delta \exp(i \zeta \cdot x/h) =0$. We also normalize the length of $\zeta $ so that 
it lies in the set $ { \cal V}$ defined by
$$ 
{\mathcal V } = \{ \zeta \in \complexes ^d : \re \zeta \cdot \im \zeta = 0, \ |\re \zeta | 
= | \im \zeta | = 1 \} \subset \{ \zeta : \zeta \cdot \zeta =0\}. 
$$
The following subset of $\mathcal{V}$ will play an important role. 
Fix $\xi \in \reals^d$ and denote
$$ 
{\mathcal V }_\xi = \{ \zeta \in \mathcal{V} : \zeta \cdot \xi=0\}. 
$$
We will assume that for some $r$ in $\{0, \dots,m-1\}$,  the amplitude $a$ 
satisfies the  transport equation of order $r+1$
\begin{equation}
\label{e:trport}
(\zeta \cdot D) ^{r+1} a =0 
\end{equation}
To construct our CGO solutions, we will study the operator obtained 
by conjugating the operator $L$ with an exponential factor. 
This leads us to consider the operator 
$$ P(hD) = P_\zeta(hD) = \exp( -ix\cdot \zeta/h) (-h^2 \Delta)
\exp(ix \cdot \zeta /h) 
= -h^2\Delta + 2 h\zeta \cdot D  
$$
where the operator $P_\zeta(hD) $ has the semi-classical symbol $p_\zeta(\xi) = 
|\xi|^2 +2\zeta \cdot \xi$. Note that we are using a slightly 
different exponential factor than in the 
earlier work \cite{MR4455267} of two of the authors. 

To construct CGO solutions as in \eqref{e:cgo}, we will need to consider the 
operator \eqref{e:OpDef} conjugated with the exponential factor
\begin{equation}
\label{Lzetadef}
\begin{split}
L _ \zeta & =\exp(-i x \cdot \zeta / h) h^{2m} L \exp(i x \cdot \zeta / h) \\
& = P_\zeta(hD)^m + h^{ 2m}\sum _ { k=0}^m \sum _ { j =0}^ k 
h^{-j} \binom k j A^{(k)} \cdot( \zeta ^j \otimes D^ { k-j} )
\end{split}
\end{equation} 
We define weak solutions of this equation in a straightforward way by relating
a solution to $L_ \zeta$ to a solution for $L$. Thus if 
$ \psi \in W^{m,2}_{loc}(\Omega)$ and $f \in \tilde  W^ { -m,2}( \Omega)$, we 
say that $\psi $ is a 
solution of $L_ \zeta \psi =f$ if we have that 
$$
h^ {2m}B_L[\exp(ix\cdot \zeta /h ] \psi, \phi)
=\langle \exp(ix\cdot \zeta /h ) f, \phi \rangle,
\quad \phi \in W^ {m,2}_c ( \Omega). 
$$

For  $\lambda \in \reals$, $h$ with  $0< h \leq 1$,  and 
$ \zeta \in {\cal V}$, we 
define spaces $X_{h\zeta}^ \lambda= X^ \lambda$ by $u \in X_{h\zeta}^\lambda$ if the Fourier 
transform of $u$ is a locally 
integrable function which satisfies 
$$ 
\| u \|^2 _ { X_{h\zeta}^ \lambda} = 
\int _ { \reals ^d} ( h + |p_ \zeta ( h\xi)| ) ^ { 2 \lambda } 
|\hat u ( \xi ) |^2 \,d \xi < \infty
$$
Note that we will often drop the subscripts $h$ or $ \zeta$ when 
the intended values may be determined from the context. 
We recall some simple estimates for the spaces $ X_{h\zeta}^\lambda$ 
from  earlier work of Brown and Gauthier \cite[Proposition 3.18]{MR4455267}. In 
the estimates below, we have $0<h\leq 1$ and $ \zeta \in {\cal V}$. 
\begin{align} 
\label{e:xembed}
&\| D^ \alpha u \| _ { X_{h\zeta}^ { \lambda _1}} \lesssim h ^ { - |\alpha | + 
\lambda _1 -\lambda _2} \| u \|_ { X_{h\zeta}^ { \lambda _2}},
&   & |\alpha | \leq 2 ( \lambda _2-\lambda _1)\\
\label{e:Xandderiv} 
 & \| u \| _ {W^ { s,2}( \reals ^d ) } \lesssim h^ { -s-\lambda} 
 \|u \|_{ X_{h\zeta}^ \lambda}, & & 0 \leq s \leq 2 \lambda, \\
\label{e:XSob}
& \| u \|_ {X_{h\zeta}^ { -\lambda }} 
\lesssim h ^ { -s-\lambda} \| u \|_ {W^ { -s, 2 }( \reals ^d) }, & 
& 0 \leq s \leq 2\lambda. 
\end{align}

The next proposition is a modification of  Proposition 3.22 of 
\cite{MR4455267} which is  in turn based on  \cite[Theorem 2.1]{MR3024091}. 
\begin{proposition} Let $ \zeta _1$ and $ \zeta_2 $ lie in $ { \cal V} $. 
    Suppose that $f$ is measurable and bounded on $ \reals ^d$, then we have 
    \begin{equation}
    \label{e:trilinear0}
    |\langle f D^ \beta u, v \rangle| \lesssim \|f\|_{ L^ { \infty} ( \reals ^d)}
    h^{-2\lambda -|\beta|} \| u\|_{ X^ { \lambda}_{h \zeta _1}} 
    \|v \|_{ X^ { \lambda}_{h \zeta _2}}, \qquad |\beta | \leq 2 \lambda.
\end{equation}
If, in addition, $f$ is uniformly continuous on $ \reals ^d$, then as $ h \to 0 ^ +$ 
we have
\begin{equation}
  \label{e:trilinear}
\begin{split} 
    |\langle D^ \alpha f D^ \beta u,v \rangle | 
\leq o(1) h^ { -2\lambda -|\alpha |-|\beta| } & \| u \| _ { X_{h\zeta_1}^{ \lambda}} 
\| v \|_{X_{h\zeta _2}^ \lambda}, \\
& 0 \leq |\alpha| + | \beta | \leq 2 \lambda , \ |\alpha | \geq 1 . 
\end{split}
\end{equation}
\end{proposition}

\begin{proof}
 We begin with the proof of the first estimate. From the Cauchy-Schwarz inequality, we have 
$$
|\langle f D^ \beta u, v \rangle |\leq \|f\|_ {L^\infty (\reals ^d)} 
\| D^ \beta u \| _ { L^2} \| v\|_{ L^2} .
$$
Our estimate \eqref{e:trilinear0} now follows by applying \eqref{e:xembed}. 

To establish the second estimate, we use a standard approximation of the identity   
$ f_ \epsilon = \phi_ \epsilon * f$, let  $ f^ \epsilon = f- f_ \epsilon $ and 
write
$$
\langle D^ \alpha f D^\beta u, v \rangle = \langle D^ \alpha f_\epsilon D^\beta u, v \rangle
+ (-1)^{|\alpha|}  \int f^ \epsilon  D^{\alpha} ( v  D^\beta u) \, dx 
$$
where we have rewritten the second term which contains the derivative of a function which may not be smooth. 
Using our estimate \eqref{e:trilinear0} in the first term and  the product rule and \eqref{e:xembed} 
in the second term, 
we have
\begin{equation*}
\begin{split}
\langle D^ \alpha f D^\beta u, v \rangle | \lesssim 
&  h^ { -2\lambda -|\beta|}\| D^ \alpha f\|_ {L^\infty(\reals ^d)} 
\|u\|_{X^ { \lambda} _{h\zeta_1}} \|v\|_{X^ { \lambda} _{h\zeta_2}} \\
& \qquad+ h^ { -2\lambda - |\alpha| - |\beta| }\|f^ \epsilon\|_ { L^ \infty( \reals ^d)}
\|u\|_{X^ { \lambda} _{h\zeta_1}} \|v\|_{X^ { \lambda} _{h\zeta_2}} . 
\end{split}
\end{equation*}
Thanks to 
our assumption that $f$ is uniformly continuous, we have 
$\|f^\epsilon \| _{L^ \infty(\reals^d)} = o(1)$ as 
$ \epsilon \to 0^+$ and the boundedness of $f$ and 
properties of $ \phi _ \epsilon$ 
give $ \|D^ \alpha f_ \epsilon \| _ {L^ \infty(\reals^d)}
= O ( \epsilon ^ { - |\alpha |})$. Using these observations gives
$$ |\langle D^ \alpha f D^ \beta u v \rangle | \lesssim ( \epsilon ^ { - |\alpha| } 
h^ { -2\lambda - |\beta|}
+ o(1) h^ { -2\lambda - |\alpha |- |\beta| } )\|u\| _ { X^ {\lambda}_{h \zeta_1}} 
\|v\|_{X^ { \lambda} _{h\zeta_2}} .
$$
If we let $\epsilon = \sqrt h$ and recall that we have $0< h \leq1$, we obtain \eqref{e:trilinear} when $|\alpha | \geq 1$. 
\end{proof}



To invert the operator $ L_\zeta$, we will assume that the components  
$A^{(k)}_\alpha$ of $A^{(k)}$ can be written in the form 
\begin{equation}
    \label{e:CoeffCond} 
    A^{(k)}_ \alpha = \sum _{ |\beta |\leq m - k } D^\beta a_{\alpha\beta},
\end{equation}
where the functions $a_{ \alpha\beta}$ are continuous in $\reals ^d$ and  
compactly supported with their support in $ \bar \Omega$. Thus  we also have that 
$A^{(k)} \in W^ { k-m, t'}( \reals ^ d)$ as in \eqref{e:coeffcond0}.  
When we give
the proof of our main theorem, Theorem \ref{t:Main}, we will see that we may
use a version of Sobolev embedding to obtain the assumption 
\eqref{e:CoeffCond} from
the hypotheses of the Theorem. 
We specialize to $ \lambda = m /2$ and assume the entries  of 
$A^{(k)}$ satisfy \eqref{e:CoeffCond}.   In this case, we have that 
$$
\| A^{(k)} \cdot D^k u \|_{X^ { -m/2}_{ h \zeta _1}} \leq o(h ^ { -2m }) \|u \|_ {X^ { m/2}_{h\zeta _2}} , \quad k= 0, \dots, m-1
$$
and more generally 
\begin{equation} 
\label{e:CfMap}
\begin{split}
\| A^{(k)} \cdot D^{k-j} u \|_{X^ { -m/2}_{h\zeta _1}} &\leq o(h ^ { -2m +j})  
\|u \|_ {X^ { m/2}_{h \zeta _2} }, \\
&k=0, \dots, m-1, \ j = 0, \dots, k.
\end{split}
\end{equation} 
In the case where $k=m$, we no longer have a constant which tends to zero with $h$ but instead have 
\begin{equation}
    \| A^{(m)} \cdot D^ { m-j}u\|_{ X^ { -m/2} _{h \zeta_1}} 
    \lesssim h^{-2m+j } \|A^{(m)}\|_{L^ \infty}\| u \| _ { X^ { m/2} _{ h\zeta_2 }}
\end{equation}
We recall that in Corollary 3.13 of earlier work \cite{MR4455267}, we found  a map 
$I_\phi: X^ \lambda \rightarrow X^ { \lambda +1}$ so that $ p(hD) I_\phi(f) = f$ 
in a neighborhood of $ \bar \Omega$. This map will satisfy the estimate
\begin{equation}
\label{e:Solvit}
    \|I_\phi (f) \| _ { X^ { \lambda +1}}\lesssim \|f\|_ { X^ \lambda}
\end{equation}
with the implied constant depending on $ \lambda$ but independent of $h$ and $ \zeta$. 
\rv{
We give the definition of the operator here, but refer the reader to our earlier
work for the proof of \eqref{e:Solvit}. 
To define the operator, we need a  fixed, smooth, compactly supported function $ \phi$ which 
is the constant 1 in  a neighborhood of $ \bar \Omega$. Using $\phi$,  the operator $I_\phi$ is 
defined using the Fourier transform by
$$
I_\phi(f) = \phi \cdot ( p_\zeta(h\cdot)^{-1}(\phi f)\hat{\ })\check{} .
$$
By using the cutoff function to restrict to the behavior on a compact set, 
we avoid having to consider the behavior at infinity which is usually treated
using weighted spaces. 
}

To establish uniqueness for our inverse problems we will
need to construct solutions for an adjoint operator $L^t$ 
as well
as for an operator $L$ and it will be an important point that the coefficients of the operator
$L^t$ satisfy the same conditions as those of $L$. The next result
makes this precise. 
\begin{proposition}
\label{p:AdjCoeff}
    If $L$ is an operator as in \eqref{e:OpDef} with 
    coefficients 
    $A^{(k)} \in \tilde W^ { k-s,p}(\Omega)$,
    $k=0, \dots, m$ and for some $s$ in $\reals$, then the 
    operator $ L^t$ also has coefficients 
    $\tilde A^{(k)} \in \tilde 
    W^ { k-s, p}(\Omega)$. 

    If the coefficients of $L$ satisfy the condition 
    \eqref{e:CoeffCond} for $k=0,\dots,m$, 
    then the coefficients of $L^t$ also satisfy this condition. 
\end{proposition}
\begin{proof}
From the definition of $L^t$, we can show that  
$$ 
L^t v = (-\Delta )^m + \sum _{ k=0}^m (-1)^k D^k \cdot ( A^{(k)}v).
$$
We use the product rule for tensors 
$$
D^k \cdot ( A^{(k)} v) 
= \sum _ { j=0}^k \binom k j (D^{k-j}\cdot A^{(k)})\cdot D^j v.
$$
Here $D^{k-j}
\cdot A^{(k)}$ is a distribution with values 
in   $S^j$ defined as follows.  If $\alpha$ is a multiindex with 
$|\alpha|=j$,  $(D^{k-j} \cdot A^{(k)})_{\alpha}
=\sum_{|\beta|=k-j} D^{\beta}A_{\alpha \beta}$
where derivatives are understood distributionally.
Collecting terms, it  follows that
the operator $L^t$ can be written as 
\begin{equation*}
L^t v = ( -\Delta )^m v 
+ \sum _{ k =0 } ^ m \tilde{A}^{( k)} \cdot  D^ k v .
\end{equation*}
with 
\[ 
\tilde{A}^{(k)}= \sum_{j=k}^m  (-1)^j \binom{j}{k} D^{j-k} \cdot A^{(j)} . 
\]

Since $D^\alpha$ maps $W^{s,p}(\reals^d)$ to
$W^{s-|\alpha|, p }(\reals^d)$, it holds that
$A^{(j)} \in \tilde{W}^{j-s,p}(\Omega)$ implies that 
$\tilde{A}^{(k)} \in  \tilde{W}^{k-s,p}(\Omega) $ as claimed. 

The calculations above also show that if  the coefficients of $L$ satisfy the condition \eqref{e:CoeffCond}, then the coefficients of $L^t$ also satisfy this condition.  
\end{proof}

\section{Constructing CGO solutions}
\label{s:CGO}
In this section we show how to solve the equation 
\begin{equation} 
\label{EqnZDef} L_ \zeta u =f, \qquad \mbox{in } \Omega
\end{equation} 
in the space $X_{h\zeta}^{m/2}$ and then obtain the existence of CGO 
solutions as in \eqref{e:cgo} with good estimates for the 
remainder in the space $X^{ m/2}_{h\zeta}$.   

\begin{proposition}
\label{prop:aprioriestimate}
Suppose that the coefficients $A^{(k)}$ of $L$  in \eqref{e:OpDef}  
satisfy \eqref{e:CoeffCond} for $k=0,\dots, m-1$ and 
that $ \|A^{(m)}\|_ \infty$ 
is sufficiently small. Then for $h$ sufficiently small, we may 
find a solution of \eqref{EqnZDef} which satisfies the estimate
$$
\|\psi \|_ { X^ { m/2}} \lesssim \|f \|_{ X^ { -m/2}}
$$
\end{proposition}

\begin{proof}
We consider the integral equation 
$$
\psi + h^{ 2m}I_\phi^m  \Big( \sum _ { k=0}^m \sum _ { j =0}^ k 
h^{-j} \binom k j A^{(k)} \cdot  (\zeta ^j \otimes D^ { k-j} \psi )\Big)= I_\phi ^m f
$$
According to the estimate \eqref{e:Solvit}, 
the right-hand side will be in $X^ { m/2}$ when $f$ lies in $X^ { -m/2}$. Using 
\eqref{e:Solvit}  and  \eqref{e:CfMap} 
(see also Proposition 3.22 in our previous work \cite{MR4455267}), 
the second term on the 
left gives a contraction on $ X^ { m /2}$ provided $h$ is small and the coefficient 
$A^{(m)}$ is small in $L^ \infty$. Thus the 
 existence of a solution to this equation follows from the contraction 
 mapping theorem.  

According to our earlier work \cite[Lemma 3.14]{MR4455267}, the solution 
of the integral 
equation is also a weak solution of \eqref{EqnZDef}.  
\end{proof}

We fix $ \zeta\in {\cal V}$ and let $ a$ be a smooth function satisfying 
one of the transport equations in 
\eqref{e:trport}. 
If we have a CGO solution as in \eqref{e:cgo}, then applying the 
operator $(-h^2\Delta)^m$ 
and multiplying by $\exp(-ix\cdot\zeta/h)$, we find that $ \psi $ is a 
solution of  the equation
\begin{multline}
    \label{e:CGOeqn}
P(hD)^m \psi + h^{ 2m}\sum_{k=0}^m A^ {(k)}\cdot (D+ \zeta/h)^k \psi \\
= - p(hD)^m a - h^ { 2m} \sum _{k=0}^m A^ {(k)}\cdot (D+ \zeta/h)^k a.
\end{multline}

When we construct CGO solutions, it will be important to 
have estimates for the righthand side of
equation \eqref{e:CGOeqn}. These estimates for the righthand side will imply estimates for
$\psi$. 

\begin{proposition} 
\label{p:AmpProp} Let $ \xi \in \reals ^d$ and $ \zeta \in {\cal V}_ \xi$.  
If $b$ is a polynomial of degree $r$, then  $a(x) = b(x) e^ { -i x\cdot \xi }$ 
satisfies the transport equation of order $r+1$ in \eqref{e:trport}.  

Suppose $a$ is smooth and satisfies the transport equation of order $r+1$
$( \zeta \cdot D) ^{r+1}  a =0$ for some $r$ in $\{0, \dots, m-1\}$, 
that $ \phi $ is a smooth cutoff function, $ \omega \in {\cal V}$ 
and  $\lambda \geq 0$, then we have
$$
    \| \phi P_ \zeta (hD)^m a\| _{X^ { -\lambda}_{h \omega}}  
    = O( h ^ { 2m - \lambda -r }).
$$
\end{proposition}

\begin{proof}
The first statement follows since $ \zeta \cdot D e^ { -ix \cdot \xi } 
= -\zeta \cdot \xi e^ {- i x\cdot \xi } $ 
and  if $b$ is a polynomial of degree $r$, then any derivative 
of order $r+1$ will be zero. 

To establish the second statement, 
    we use the binomial theorem to expand
    $$
    P(hD)^m a= \sum_{ k =0 }^{r}\binom m k (-h^2\Delta )^{m-k} (2h\zeta \cdot D)^k a.
    $$
    Since $a$ satisfies the transport equation of order $r+1$ 
    (see \eqref{e:trport}), the terms 
    with $ k=r+1, \dots, m$ vanish and hence the upper limit in the above sum is $r$. 
    Each of the  nonzero  terms contain $h$  raised to a power that 
    is at least $2m -r$.  Since $a$ is smooth, 
    this provides an easy estimate for the $L^2$-norm on compact 
    subsets of $ \reals ^d$. We use that 
    $ (h + |p_\omega(h \xi) |) ^ { -\lambda } \leq h^{-\lambda} $ 
    to obtain the estimate for the 
    $ X^ { -\lambda}$-norm from the estimate for the $L^2$-norm 
    (or see \eqref{e:xembed}). Note that we do not need any relation 
    between the vectors $ \zeta$ in the transport equation and $\omega$ 
    in the $X ^{-\lambda }$-norm. 
\end{proof}

A key step in our proof involves taking averages of the coefficients as we allow $h$ and $ \zeta$ to vary
and obtaining a good estimate for these averages.  
This estimate allows us to choose  values $h$ and   $ \zeta$ for
which the terms involving the coefficients on the righthand side of \eqref{e:CGOeqn} are small. 
The result below is from work of 
Brown and Gauthier \cite[Theorem 4.6]{MR4455267} and is an extension of  a 
result from Haberman and Tataru \cite{MR3024091}. 

The statement of Theorem 4.6 in Brown and Gauthier's work omits 
the condition equivalent to  $s-k -2 \sigma \geq 0$ in the result below, but 
this condition is needed in our earlier work. This is 
the error mentioned in the introduction. 

\begin{proposition}
\label{p:Avg}
Suppose that $\sigma$, 
$k$, $s$, and $\lambda$ satisfy  $0 \leq \sigma < 1$
\begin{equation}
\label{SupCond}
(2\lambda  - 2 \sigma)  \geq s-k -2 \sigma \geq 0
\end{equation}
Fix $ \zeta \in {\cal V}$ and let $ \zeta( \theta) 
= e^ { i \theta } \zeta $, $\theta \in \reals$ 
which will also lie in $\cal V$.  
Let  $f \in W^{k-s,2}(\reals^d)$
$$
\frac 1 h \int _ 0 ^ {2\pi } \int _h ^ { 2h} 
\| f\|^2_{X^ { -\lambda}_{\tau \zeta ( \theta)} }
d\tau \, d\theta \lesssim h ^ { 2(-\lambda -s + k + \sigma)} 
\| f\|^2_ { W^ { k-s, 2} ( \reals ^ d)}
$$
\end{proposition}
\begin{proof}
We write 
\begin{multline}
\label{e:avgstp1}
\frac 1 h \int _ 0 ^ {2\pi } \int _h ^ { 2h} 
\| f\|^2_{X^ { -\lambda}_{\tau \zeta (\theta)}}
\, d\tau \, d\theta \\
=
\frac 1 h \int _ 0 ^ {2\pi }\!\! \int _h ^ { 2h} \int_{ \reals ^ d} 
| \widehat {f} | ^ 2 
( \tau + |p_{\zeta(\theta)} (\tau \xi) |) ^ { - 2\lambda } \, d\xi \,  d\tau \, d\theta
\\
\leq \frac 1 h \int _ 0 ^ {2\pi } \int _h ^ { 2h} \int_{ \reals ^ d} | 
\widehat {f} | ^ 2 
(\sup _ { \xi } \frac { \langle \xi \rangle ^ {2s-2k-4 \sigma}} 
{( \tau + |p_{\zeta(\theta)} 
(\tau \xi) |) ^ {  2\lambda- 2 \sigma  }} )
\frac { \langle \xi \rangle^ { 2k-2s+ 4 \sigma}}
{ ( \tau + |p_{\zeta( \theta) }( \tau \zeta )| ) ^ { 2 \sigma} }
\, d\xi \,  d\tau \, d\theta
\end{multline}
From  condition \eqref{SupCond} we obtain
$0 \le 2(s-k-2\sigma) \le 4\lambda-4\sigma$.
Thus for  small enough $h$, and $h \le |\tau| \le 2h$,   we can 
apply an estimate of  Brown
and Gauthier 
\cite[equation (3.17)]{MR4455267}
to obtain that 
$$ 
\sup _{ \xi \in \reals ^d } \frac { \langle \xi \rangle ^ {2s-2k-4 \sigma}} 
{( \tau + |p_{\zeta(\theta)}
(\tau \xi) |) ^ {  2\lambda- 2 \sigma  }} 
\lesssim \tau^{-2\lambda-2s+ 2k + 6 \sigma} 
\lesssim h^{-2\lambda-2s+ 2k + 6 \sigma}
$$
Substituting this displayed inequality into \eqref{e:avgstp1} gives 
\begin{multline}
\label{e:RTA}
\frac 1 h \int _ 0 ^ {2\pi } \int _h ^ { 2h}
\| f\|^2_{X^ { -\lambda}_{\tau \zeta ( \theta)} } d\tau d\theta
\lesssim
 h^{-2\lambda-2s+ 2k + 6 \sigma} \\ 
 \times \frac 1 h \int _ { \reals ^d}| \widehat {f} (\xi)|^2
 \langle \xi \rangle^ { 2k-2s+ 4 \sigma}
 \int _ 0 ^ { 2\pi } \int _ h ^ { 2h} \frac 1 {( \tau +| p_{\zeta( \theta)}
 ( \tau \xi ) |) ^ { 2 \sigma} }\, d\tau \,  d\theta \, d\xi . 
\end{multline}
Proposition 4.5 in Brown and Gauthier \cite{MR4455267}
gives that for $ 0 \leq \sigma < 1$
$$
\frac 1 h \int _ 0 ^ { 2\pi } \! \int _ h ^ { 2h} 
\frac 1 {( \tau + |p_{\zeta( \theta)} ( \tau \xi) | ) ^ { 2 \sigma}} \, d\tau\,  d\theta
\lesssim \frac 1 { h^ { 4 \sigma} \langle \xi \rangle ^ { 4 \sigma}}
$$
Putting this estimate into \eqref{e:RTA} gives the conclusion of the Proposition. 
\end{proof}

\begin{proposition} 
\label{p:Avgo1}
Let $ \zeta \in { \cal V}$,   $ 0 \leq \sigma < 1$ and 
    suppose $2\lambda - 2\sigma \geq s-k -2  \sigma \geq 0$ and that  
    $ s-k >0 $ if $ \sigma =0$. 
    Given $ f \in W^{k-s,2} (\reals ^d)$, we may find a sequence $\{(h_n, \zeta _n= e^ { i \theta _n} \zeta) \} $ with 
    $ \lim _{ n \to \infty } h _ n =0$ and $\lim _ { n \to \infty} \zeta _n =  \zeta_0 = e^ { i \theta _0} \zeta$ so that 
    $$
    \| f\|_{ X^ { - \lambda}_{h_n \zeta_n}}= o ( h _n^{ -\lambda + k -s + \sigma}). 
    $$
\end{proposition}
\begin{proof}
    We will divide $f$ into high and low frequency pieces. 
    For $N>0$ we define $ f_N$ by $ \hat f_N ( \xi) 
    = \chi _{[0,N]}( |\xi|)\hat f (\xi)$ and $ f^N = f- f_N$.
  
  Given $N$, we use Proposition \ref{p:Avg} and Chebyshev's inequality to find a sequence $\{(h_n, \zeta_n)\}$ 
   with $h_n$ tending to zero and $ \zeta _n = e ^ { i \theta _n } \zeta $  so that for $(h_n, \zeta_n)$ in this sequence, we have 
    $$
      h_n ^ {\lambda +s-k - \sigma}  \| f^N\|_{ X^ { - \lambda  }_{ h_n\zeta_n  } } \lesssim 
      \| f ^ N\|_{ W^ { k-s, 2} ( \reals ^ d ) }. 
    $$

The function $f_N$ will be in $L^2$ and we may use      \eqref{e:XSob}
to find that for any $\zeta$, 
    $$ 
  h^{\lambda - k + s-\sigma}  \| f_N \| _ { X_{h\zeta}^ { -\lambda } }
    \leq   \| f_N\|_ {L^ 2} h ^ { s-k-\sigma}. 
    $$
    Our assumptions imply $ s-k - \sigma > 0$ so that  the righthand side  will tend to zero with $h$. 

   Combining our observations, for each $N$, we have a 
   sequence $\{(h_n, \zeta_n)\}$
   so that
   $$
 h_n ^ { \lambda +s-k - \sigma} \|f\| _ {X^ {-\lambda}_{h_n\zeta_n}} \lesssim 
  ( \|f^N\| _ { W^ { k-s, 2} ( \reals ^d)} 
   +  \|f_N \|_ {L^2  ( \reals ^d) } h_n ^ { s-k+ \sigma}  ).  
   $$
   Taking a limit as $n $ approaches $\infty$, we have 
   $$
 \limsup_{ n\to \infty } h _n ^ { \lambda +s-k -\sigma}   \|f\| _ {X^ {-\lambda}_{h_n\zeta_n}} \lesssim 
 \|f^N\| _ { W^ { k-s, 2} ( \reals ^d)} 
  . 
   $$
   Since $f\in W^ { k-s, 2}( \reals ^d)$, we have $\lim _{ N\to \infty} \|f^N\| _ { W^ { k-s, 2} ( \reals ^d)} =0$. 
   A diagonal argument allows us  to obtain a sequence which satisfies the conclusion of the Proposition and since $ \zeta_n = e^ { i\theta_n} \zeta$
   we may choose a subsequence so that $ \zeta_n $ converges. 
\end{proof}

In addition to \eqref{e:CoeffCond} we  assume that the coefficients $A^{(k)}$ lie in a Sobolev 
space. To be precise, we require  for some $s$ with $ s<m$
\begin{equation}
    \label{e:CoeffSobCond}  
    A^{(k)} \in \tilde W^{k-s,2}( \Omega). 
\end{equation}
From the assumption \eqref{e:CoeffSobCond} we may use  Proposition \ref{p:Avg}  (see also \cite[Theorem 4.6]{MR4455267})  
and \eqref{e:CoeffSobCond}, to  find  a sequence $\{(h_n, \zeta_n)\}$ so that  
\begin{equation}\label{e:CoeffEst}
\| A^ {(k)}\|_{X^ { (k-m)/2 }_{ h_n \zeta_n } } =o( h_n^ { 3k/2 -s- m/2 + \sigma }),\qquad 0\leq k \leq k_0. 
\end{equation}
provided $ \sigma \in [0,1)$, $s-k_0 -2\sigma\geq 0$, $m\geq s$ and $s-k_0>0$. 
As in the Proposition, we have  $\lim_{ n\to \infty } h_n =0$ and $ \zeta_n = e^{i\theta_n}$ with 
$\lim _{ n\to \infty } \zeta_n = e^ { i \theta _0} \zeta$. 

In our application, we will need to find one sequence so 
that the estimate \eqref{e:CoeffEst} holds for  the coefficients 
for an  operator $L_1$ using norms depending on $\zeta $ and  the 
coefficients of $L_2^t$
with norms defined using  $-\zeta$. 
While this is a rather lengthy set of 
conditions, it is a finite set and we can still  apply Chebyshev's inequality to 
obtain one sequence with the  desired behavior. We will refer to this sequence as our
magic sequence. 

\begin{proposition} 
\label{p:CGOestimate}
Consider the operator $L$ from \eqref{e:OpDef}, let $k_0\leq m-1$  and
assume that  $A^{(k)} =0$ for $ k > k_0$. 
Suppose that for $k=0, \dots, k_0$, the coefficients $ A^{(k)}$   satisfy \eqref{e:CoeffCond}
and  \eqref{e:CoeffSobCond} for 
some $s$ with $s \leq m$.
Assume that $s-k_0>0$ and choose $ \sigma $ with $ 0 \leq \sigma <1$ so that $ s-k_0 - 2 \sigma \geq 0$.

Choose $ \zeta \in {\cal V}$ and let $a$ be an amplitude satisfying the transport equation 
$(\zeta \cdot \nabla)^ { r+1} a =0$ for some $r$ in $ \{0, \dots, m-1\}$. 

Under these conditions, we have a sequence $\{(h_n,\zeta_n)\}$ 
with $\lim _{ n\to\infty} (h_n, \zeta_n) = (0, e^ { i \theta_0} \zeta)$  and 
so that for $(h_n, \zeta_n)$ 
in this sequence, we may find a CGO  solution as in \eqref{e:cgo} 
and  we have the estimate 
\begin{equation}
\label{e:SolEstimate}
    \|\psi \| _ { X^{m/2}_{h_n \zeta_n } }\lesssim 
    ( h_n ^ { 3m/2 - r } + h_n^ { 3m/2 -s + \sigma}).
\end{equation}
\end{proposition}

\begin{proof}
    We will use Proposition \ref{prop:aprioriestimate} to find a solution to \eqref{e:CGOeqn} 
    in a neighborhood of $ \Omega$. Since we only
    hope to find a solution near $\bar \Omega$, we may multiply the right-hand side 
    of \eqref{e:CGOeqn} by a cutoff function $\phi$ which 
    is one in a neighborhood of
    $ \Omega$ and thus we want to estimate 
    \begin{equation}
    \label{e:TruncRHS}
     -\phi ((p(hD)^m a + h^ { 2m} \sum _{k=0}^{k_0} A^ {(k)}\cdot (D+ \zeta/h)^k a)
     \end{equation}
     in the space $X^ { -m/2}$.
   Since $r\leq m-1$, we may use  Proposition \ref{p:AmpProp}
to conclude 
    \begin{equation}
    \label{e:AmpTerms}
        \| \phi P(hD)^m a \| _{ X^ { -m/2}} \lesssim h ^{3m/2 - r}.
    \end{equation}

    To estimate the terms involving 
    the coefficients, we  use 
     \eqref{e:CoeffEst} to find a sequence  
    $(h_n,\zeta_n)$ for which  we have 
    \begin{equation}
    \label{e:CoeffTerms}
    h_n^{2m}\| A^{(k)} \cdot ( D + \zeta /h_n)^k a\|_{X^ { -m/2}} \lesssim o(h_n^{3m/2 + k-s + \sigma})
    \cdot h^  { -k} = o(h_n^{ 3m/2 -s + \sigma}). 
    \end{equation}
    Note that these estimates use that $m> s-k - 2\sigma \geq 0$, $ 0 < h \leq 1$, 
the smoothness of $a$, that $A^{(k)} $ is  
    compactly supported and the fact that the spaces $ X^ \lambda$ are preserved by 
    multiplication by smooth functions.  From the observations \eqref{e:AmpTerms} 
    and \eqref{e:CoeffTerms}, we have that the $X^ { -m/2} $-norm of the  
    expression \eqref{e:TruncRHS} is bounded by  $(  h ^ { 3m/2 -r} +h^ { 3m/2 - s+\sigma} )$. The  
    desired estimate for the solution $ \psi$ now follows from Proposition \ref{prop:aprioriestimate}. 
\end{proof}

\rv{
Before continuing, we point out that Proposition \ref{p:CGOestimate} 
presents the key improvement
that we obtain from the averaging argument of Haberman and Tataru. The averaging
argument is responsible for the factor $h^\sigma$ that appears in the estimate for
the remainder term \eqref{e:SolEstimate} and  in Proposition \ref{p:Avgo1}. This 
term is what allows 
 us to study less regular conductivities.
Without this improvement we would end up proving results similar to those established
by Assylbekov and Iyer \cite{MR4027047}. 
}

\begin{remark} Proposition \ref{p:Avgo1} gives an asymptotic statement involving 
$o(h^K)$ for some power $K$. If we take advantage of this in our proof of Proposition \ref{p:CGOestimate}, we can obtain  the slightly better estimate $ \|\psi \|_{X^{m/2}_{h\zeta} }
\lesssim (h^{ 3m/2 -r}  +o(1)h ^{ 3m/2 -s+ \sigma})$.  We do not pursue this 
since the improvement does not give a better conclusion in our main theorem. 
\end{remark}

\section{From the Bilinear Form to equations  on the coefficients} 
\label{s:maineqn}
We suppose we have two operators $L_1$ and $L_2$ of the form \eqref{e:OpDef} 
and 
we let $A^{(k)}_\ell $ denote the coefficients for $L_\ell$.
We continue to  let $k_0 \leq m-1$ be the 
index for which  $ A_\ell ^ {(k)}=0$ for $k > k_0$.

The uniqueness theorem we hope to prove begins with a hypothesis on the bilinear form 
\[ 
   B[u_1,u_2]= B_{L_1}[u_1, u _2]- B_{L_2}[u_1, u _2] 
   = \sum_{k=0}^{k_0} \langle (A^{(k)}_1 - A^{(k)}_2)\cdot D^k  u_1 , u_2 \rangle. 
\]
We say that the forms $B_1$ and $B_2$ are equal if whenever $u_1 $ is a weak solution of $L_1 u_1=0$ in $\Omega$ 
and $u_2$ is a weak solution of $L_2^t u_2=0$, then 
\begin{equation}
\label{e:form}
    B[u_1, u_2]=0. 
\end{equation}
In what follows, we will let $ A^{(k)}= A_1^ {(k)}- A_2^{(k)}$ in order to simplify
our notation. 
See \cite[Section 5]{MR4455267} for further discussion of the assumption \eqref{e:form}
and its relation to more traditional 
hypotheses on a  Dirichlet to Neumann map or Cauchy data. 

\begin{theorem}
\label{t:MEHolds}
Suppose that the coefficients  $ A_\ell ^{(k)}$,  $\ell=1,2$, $k=0, \dots, k_0$
satisfy \eqref{e:CoeffCond}  and \eqref{e:CoeffSobCond}.
Suppose that the form  $B=B_{L_1}-B_{L_2}$ satisfies 
the condition \eqref{e:form}. 
Let $\xi \in \reals ^d$ and suppose that 
$ \zeta  \in {\cal V}_\xi$, $ \omega \in \complexes ^d$. 
We set $ A^ {(k)}= A_1^{(k)}-A_2^{(k)}.$
  Assume that we are in one of the two cases of Theorem \ref{t:Main}: 
  a) $k_0=\lfloor m/2 \rfloor -1$ 
  and $m/2 < s<  m/2 +1$ or b) $m$ odd, $k_0=m/2 -1/2$ and 
  $s= m/2 + 1/2$.  Then for $r_\ell\leq m/2$, we have   
\begin{equation} 
\label{e:MEHolds}
\begin{aligned}
\sum _ { k = 0}^ { k_0-j} \binom {k+j} j 
\langle A^{(k+j)} \cdot 
(\zeta ^j \otimes D^ { k} )  (\omega \cdot x)^{r_1} , 
( \omega \cdot x) ^ {r_2} e^ { -i x\cdot \xi } \rangle =0, \\
\quad j = 0,\dots , k_0.
\end{aligned}
\end{equation}
\end{theorem}

\begin{proof} In the first case we define $\sigma$  to be  $ \sigma = s - m/2$  and
we put 
$ \sigma =1/2$ in the second case where $m$ is odd
and $k_0 = m/2 - 1/2$.
Our amplitudes will be  $ a_1(x) = ( \omega\cdot x ) ^ { r_1}$ and 
$ a_2 (x) = e^ { -ix\cdot \xi} (\omega \cdot x) ^ { r_2}$. 
These amplitudes will satisfy 
 the transport equations, $ (\zeta \cdot D)^ { r_\ell +1} a_\ell =0$ and 
 since $r_\ell \leq m/2$, we will have $r_\ell \leq m-1$ as in 
 Proposition \ref{p:CGOestimate}. 
With these choices, we have $s-k_0 -2\sigma \geq 0$, so we may  use 
Proposition \ref{p:CGOestimate} to find 
 $u_1 (x) = e^ { ix \cdot \zeta /h} ( a_1(x)+ \psi_1(x))$ 
 a CGO solution (see \eqref{e:cgo}) of $ L_1 u _1=0$,
$u_2(x) = e^ { -ix\cdot \zeta /h} (a_2 + \psi_2(x)) $ a 
CGO solution of $ L_2^t u_2 =0$ and the corresponding magic sequence 
$(h_n, \zeta_n)$ for
which we have the estimate \eqref{e:SolEstimate}. Here, we have used 
Proposition \ref{p:AdjCoeff} to see that the coefficients of $L_2^t$ 
satisfy the conditions of
Proposition \ref{p:CGOestimate}. 
 
Substituting these solutions into the form $B$ and using the condition 
\eqref{e:form}, we have
$$
\sum _ { k=0}^{k_0} \langle A^{(k)}
  \cdot (D+ \zeta /h) ^k (a_1 + \psi_1) , (a_2 + \psi _2)\rangle = 0.
$$
Collecting like powers of $h$ we may rewrite the previous displayed equation as 
\begin{equation}
\label{e:zeroform}
\begin{split}
&\sum _ { k=0}^{k_0}  \langle A^{(k)} 
  \cdot (D+ \zeta /h) ^k (a_1 + \psi_1) , (a_2 + \psi _2)\rangle  \\ 
&= \sum _ { k=0}^{k_0} \langle A^{(k)} 
\cdot ( \sum _ { j = 0} ^k \binom k j h^{ -j} ( \zeta ^j 
\otimes D^ { k-j} )) (a_1 + \psi_1) , (a_2 + \psi _2)\rangle \\
 & = \sum_{j=0}^{k_0} h^{ -j} \sum _ { k = j}^ { k_0} \binom k j \langle 
 A^{(k)} \cdot 
(\zeta ^j \otimes D^ { k-j} ) (a_1 + \psi_1) , (a_2 + \psi _2)\rangle  \\
& = \sum _ { j=0}^{k_0} h^{ -j} \sum _ { k = 0}^ { k_0-j} \binom {k+j} j 
\langle A^{(k+j)} \cdot 
(\zeta ^j \otimes D^ { k} ) (a_1 + \psi_1) , (a_2 + \psi _2)\rangle = 0.
\end{split}
\end{equation}

Then combining \eqref{e:SolEstimate}  with \eqref{e:CoeffEst}, we  obtain 
the  following estimates for the  terms involving $\psi_\ell$, $\ell=1,2$, where we 
assume $k\geq 0$, $j \geq 0$ and  $ k+j \leq k_0$
\begin{equation}
\label{e:apsi}
\begin{split}
  h^ { -j} | \langle A ^ {( k+j)}  \cdot &( \zeta ^ j 
  \otimes D^k )a_1, \psi _2 \rangle  | 
    \lesssim h^{-j} \| A ^ {( k+j)}\|_{X^ { -m/2}}
   \| \psi _2\| _ { X^ { m/2}} \\
   & \lesssim o(h^{ -s-m/2 + \sigma+ k})
   ( h ^ { 3m/2 - s + \sigma }+ h^ { 3m /2 - r _2 }). 
\end{split}
\end{equation}
Where we have used \eqref{e:XSob} to see that $\|A^{(k+j)}\|_{X^{-m/2}}\lesssim
h^{-(k+j)/2}\|A^{(k+j)}\|_{X^{(k+j-m)/2}}$. 
Similarly, we have
\begin{equation}
\label{e:psia}
\begin{split}
h^ { -j} \langle A ^ {( k+j)}\cdot ( \zeta ^ j \otimes D^k)
\psi_1, a_2 \rangle 
&\lesssim h ^ { -j} \| A ^ {( k+j)}\| _{ X^ { ( k-m)/2}} 
  \| D^ k \psi _1 \| _{ X^ { ( m-k ) /2 }}  \\
  &   \lesssim o(h^{ -s - m/2 + \sigma})
  ( h ^ { 3m/2 - s + \sigma }+ h^ { 3m /2 - r _1 }). 
    \end{split}
\end{equation}
In the proof of  \eqref{e:psia} we use \eqref{e:xembed} 
to estimate  $ D^k \psi_1$ in the space $X^{(m-k)/2}$ in terms of the $X^{m/2}$ norm of $\psi$.

For the  estimate of the terms involving both $\psi_1$ and $\psi_2$, we use that the coefficients $A^{(k)} $ have the 
representation \eqref{e:CoeffCond}.  With 
this assumption and \eqref{e:trilinear},  we have 
\begin{equation}
\label{e:psipsi}
\begin{split}
h^{-j}\langle A^ { ( k+j)}  & \cdot (\zeta^ j \otimes D^ k)\psi _1, \psi _2 \rangle \\
& \lesssim o(1) h ^ { -2m }( h^ { 3m/2 - s + \sigma } + h^ {3m/2 -r_1} ) ( h^ { 3m/2 - s + \sigma } + h^ { 3m/2 - r _ 2} )
\end{split}
\end{equation}
where we use once more \eqref{e:SolEstimate}  to estimate $\| \psi_l\|_ { X^ { m/2}} $.
Expanding the right-hand sides of the estimates (\ref{e:apsi}-\ref{e:psipsi}) we see that these terms will go to zero
if we have 
\begin{align}
m -2s + 2\sigma &\geq 0 \label{e:cond1} \\
m -s+ \sigma  &\geq r_\ell, \quad \ell = 1,2 \label{e:amp} \\
    m   &\geq  r_1+ r_2  \label{e:ampamp}. 
\end{align}
The first condition \eqref{e:cond1} follows from our choice of $\sigma$. Given \eqref{e:cond1}, 
we have $m/2 \geq s-\sigma$ which implies that $ m-s+ \sigma \geq m/2$ 
and since $r_\ell \leq m/2$, we have \eqref{e:amp}.  The third will 
follow since we assume $ r_\ell \leq m/2$. 
Using the estimates (\ref{e:apsi}--\ref{e:psipsi}), we have  that 
for $j= 0, \dots, k_0$
\begin{equation*}
\begin{split}
\sum _ { k = 0}^ {k_0-j} \binom {k+j} j \big(\langle A^{(k+j)} &\cdot 
(\zeta ^j \otimes D^ { k} )(a_1 + \psi_1) , (a_2 + \psi _2)\rangle \\
&-
\langle A^{(k+j)} \cdot 
(\zeta ^j \otimes D^ { k})a_1  , a_2 \rangle \big)
= o(h_n ^j )
\end{split}
\end{equation*}
We take the last line of \eqref{e:zeroform}, multiply by $h^{k_0}$ and take 
the limit along our
sequence $\{(h _n, \zeta_n)\}$. Based on the estimates (\ref{e:apsi}--\ref{e:psipsi}), 
we can show that 
\begin{equation}
\begin{split}
\lim _ { n  \to \infty }  
h_n^{ k_0}\sum _ { j=0}^{k_0} h_n^{ -j} \sum _ { k = 0}^ {k_0-j} \binom {k+j} j \langle A^{(k+j)} \cdot 
(\zeta_n ^j \otimes D^ { k} ) (a_1 + \psi_1) , (a_2 + \psi _2)\rangle \\
= e^ { i k _0 \theta_0 } \langle A^{(k_0)}   \cdot \zeta^ { k_0 } a_1, a_2 \rangle . 
\end{split}
\end{equation}
And since \eqref{e:zeroform} holds we obtain that the right-hand side of 
the above equation is zero. 
This gives \eqref{e:MEHolds} in the conclusion of this Theorem for $ j = k_0$. 

We now give an inductive argument to show that if \eqref{e:MEHolds}
holds
for 
$ j\geq j_0+1$, then it also holds for $ j= j_0 $. We again consider the 
last line in \eqref{e:zeroform} and subtract an expression which we 
know is zero thanks 
to our induction hypothesis
\begin{multline}
 \sum _ { j=0}^{k_0} h^{ -j} \sum _ { k = 0}^ { k_0-j} \binom {k+j} j \langle 
 A^{(k+j)} \cdot 
(\zeta ^j \otimes D^ { k} ) (a_1 + \psi_1) , (a_2 + \psi _2)\rangle \\
- \sum _ { j =j_0+1}^ { k_0} h^{ -j} 
\sum _ { k = 0}^ { k_0-j} \binom {k+j} j 
\langle A^{(k+j)} \cdot 
( \zeta ^j \otimes D^ { k} ) (\omega \cdot x)^{r_1} , 
( \omega \cdot x) ^ {r_2} e^ { -i x\cdot \xi } \rangle =0. 
\end{multline}
Now we multiply by  $h^ { j_0} $ and let $h$ and $ \zeta $ be in our 
magic sequence.   
Thanks to the estimates (\ref{e:apsi}-\ref{e:psipsi}) and the assumptions 
(\ref{e:cond1}-\ref{e:ampamp}) 
the terms involving $\psi_\ell$ tend to zero and the limit along 
this sequence gives  the equation
$$
e^{ i \theta_0 j_0} \sum _ { k = 0}^ { k_0-j_0} \binom {k+j_0} {j_0} 
\langle A^{(k+j_0)} \cdot 
(\zeta ^{j_0} \otimes D^ { k}  ) (\omega \cdot x)^{r_1} , 
( \omega \cdot x) ^ {r_2} e^ { -i x\cdot \xi } \rangle =0.
$$
Which gives the conclusion of our inductive step. 
\end{proof}

\section{Uniqueness from the main equation}
\label{s:uniq}

In this section, we show how \eqref{e:MEHolds} in  Theorem \ref{t:MEHolds} 
can be used to show that  the difference of the coefficients 
$ A^{(k)}_1- A_2^{(k)}$ is zero. We will call \eqref{e:MEHolds} the main equation and 
will continue to use $ A^{(k)}$ to 
denote the difference $ A^{(k)} = A^{(k)}_1-A_2^{(k)}$.  
\begin{theorem}
\label{thm:uniquenessmaster} 
Let $\{A^{(k)}\}_{k=0}^{k_0}$ be a  
collection of compactly supported $k$ tensor valued distributions. Suppose that  such that for
$j=0,\dots,k_0$ the following equations hold: 
\begin{equation}
\label{eq:conditionmaster}
\sum _ { k = 0}^ { k_0-j} \binom {k+j} j \langle A^{(k+j)} \cdot 
(\zeta ^j \otimes D^ { k}  a_1) , a_2\rangle=0, \quad 
\end{equation}
Here, we let  $\xi \in \reals^d$,   $\zeta \in {\cal V}_\xi$, $\omega \in \complexes ^d$, 
$\ell=1,2$, the amplitudes
$a_1(x) =  (\omega \cdot x)^{r_1} $,  $a_2(x) =  (\omega \cdot x)^{r_2} e^ { -ix\cdot \xi} $, 
$r_\ell$ are nonnegative integers such that $r_1+r_2 \le  R $, 
$\zeta \in \mathcal{V}_\xi$
and $\omega \in \complexes^d$. Then 
$$
  \hat A^{(k)}(\xi) =
    \begin{cases} 
     0, \qquad   & k \leq R \\ 
      \xi ^ { s+1} \otimes B^ {(k-s-1)}(\xi),   \qquad & R+1\leq k  \leq 2R+1
    \end{cases}
$$
where $ s= 2R+1 -k$. 
\end{theorem}
\begin{remark}
    The condition that $ \hat A^{(k)}(\xi) = \xi ^{s+1} \otimes B^ { (k-s-1)}(\xi)$ can be formulated 
    using the symmetric derivative as defined in Sharafutdinov \cite[p.~25]{MR1374572}, for 
    example. If $B$ is an $S^k(\complexes^d)$-valued function on $ \reals^d$, we will 
    denote the symmetric derivative of $B$ by $D\otimes B = \sigma(DB)$ which is an 
    $S^{k+1}(\complexes^d)$-valued function given as
    the symmetrization of the array $ DB$ with entries 
    $(DB)_{j\alpha}= D_j B_\alpha$.  Note that since 
    we are using $D$ rather than $ \nabla$, our 
    operator differs from Sharafutdinov's by a 
    factor of $i$.  With this 
    notation our theorem states that 
    $ A^{(k)}= D^{s+1} \otimes \check B^{(k-s-1)}$ when $ R+1\leq k \leq 2R+1$. 
\end{remark}

The proof relies on several lemmata.  A key step is the characterization of the tensors $B$ 
which satisfy $B\cdot \zeta^k=0$,     $ \zeta \in {\cal V}_\xi$.  
This characterization is the content of 
Theorem \ref{thm:tensorstructure} in subsection \ref{sec:tensorstructure} 
which we call our tensor structure theorem.  In  subsection \ref{s:derivcond}
we use the tensor structure theorem  and  \eqref{eq:conditionmaster}  
to prove the uniqueness  result, 
Theorem~\ref{thm:uniquenessmaster}. We will see that \eqref{eq:conditionmaster} 
can be interpreted as a condition on
directional derivatives of the Fourier transform of the coefficient. 
The third subsection  puts 
these two 
facts together to prove Theorem~\ref{thm:uniquenessmaster}. We then 
prove Theorem \ref{t:Main} which follows easily given the tools we have 
developed.   A final section lists some questions raised by 
this research that we have not been able to answer. 

\subsection{Tensor Structure theorem}\label{sec:tensorstructure}

Let $I_2$ be the tensor of order 2 given by 
$(I_2)_{jk}=\delta_{jk}$,  $\zeta \in \complexes^d$ 
such that $\zeta \cdot \zeta=0$ and let 
$B^{(k-2)}$ be an arbitrary $(k-2)$-tensor. Using \eqref{e:id1} we see that 
\begin{equation}\label{eq:Bk-2}
 (I_2 \otimes B^{(k-2)}) \cdot \zeta^k 
 = (\zeta \cdot \zeta) (B^{(k-2)} \cdot \zeta ^ { k-2})=0.
\end{equation}
Similarly, if $\zeta \cdot \xi=0$ and 
 $B^{(k-1)}$ is $(k-1)$-tensor, we can use \eqref{e:id1} to conclude 
\begin{equation}\label{eq:Bk-1}
(\xi \otimes B^{(k-1)} ) \cdot \zeta^k = (\zeta \cdot \xi) 
(B^ { (k-1)} \cdot \zeta ^ { k-1})=0.
\end{equation}
In Theorem~\ref{thm:tensorstructure} we show that  
any tensor $A \in S^k ( \complexes ^d)$ for 
which $ A\cdot \zeta ^k=0$ whenever $ \zeta \in { \cal V}_\xi$ is
a sum of terms as in \eqref{eq:Bk-2} and \eqref{eq:Bk-1}. 


For the theorem below and later we introduce the convention that a term involving  
a tensor of negative order is zero. For example, in the theorem below, when $ k=1$, we will show that the 
term  $I_2\otimes B^{(k-2)}=0$. 

\begin{theorem}[Tensor structure theorem] 
\label{thm:tensorstructure} 
Let $\xi \in \reals ^d$ and suppose that 
$A^{(k)} \in S^k(\complexes^d)$ with $ k \ge 1$ and that 
$$
    A^{(k)}\cdot \zeta ^k =0, \qquad \zeta \in { \cal V} _\xi . 
$$
Then,  there exist tensor valued functions $B^{(k-2)}(\xi) \in S^{k-2}(\complexes^d)$ 
and $B^ { (k-1)}(\xi) \in S^{k-1}(\complexes^d)$ such that 
\[
    A^{(k)} = I_2\otimes B^ { (k-2)}(\xi) + \xi\otimes B^ { (k-1)} (\xi). 
\]
\end{theorem}

We begin the proof of  Theorem~\ref{thm:tensorstructure}  begins by considering tensors
for which $ A^{(k)}\cdot \zeta ^k=0$ for all $ \zeta \in {\cal V}$ and in 
Proposition \ref{prop:tensoronV}
we show that such a tensor is of the form 
$ I_2 \otimes B^ { (k-2)}$.  We apply this result on a subspace of 
codimension one to obtain the theorem. 

\subsubsection{Tensors vanishing on  ${\cal{V}}$}

\begin{proposition}\label{prop:tensoronV} Let $k\ge 2$. Suppose that 
$A^{(k)} \in S^k(\complexes^d)$  and that  
\begin{equation}
\label{eq:condition}
A^{(k)} \cdot  \zeta^k=0, \quad \zeta \in { \cal V}. 
\end{equation}
Then we may find  $B^{(k-2)}$ a tensor 
in 
$ S^{k-2}(\complexes^d)$ so that 
$ A^{(k)} = I_2\otimes B^ {( k-2)} $. 
If $k=1$, then $A=0$. 
\end{proposition}
\begin{proof}
 Let $k\ge 2$.
The proof relies on dimension counting and thus it is combinatorial in spirit. Since every tensor of the form 
$I\otimes B^ { (k-2)}$ satisfies that $(I\otimes B^ { (k-2)}) \cdot  \zeta^k=0$ for $\zeta \in {\cal V}$, it is 
enough to prove that the complex vector space of tensors such that  
$A^{(k)} \cdot \zeta^k=0$ has the dimension of the space of  $k-2$ symmetric tensors, $S^ { k-2} ( \complexes ^d)$. 
Notice that the dimension of  $S^ k ( \complexes ^d)$ is  
the same as the dimension  of the vector space of homogeneous polynomials of degree $k$ in $d$ variables. 
It is well-known that the dimension of this space equals
\begin{equation}
    \label{eq:degreek}
 \dim (S^k( \complexes^d))=    \binom{k+d-1}{d-1}. 
\end{equation}
See, for example, the textbook  by 
Blitzstein and Hwang \cite[Example 1.4.22]{MR3929729}. 
We will also need the number of monomials in $d$ variables of degree less or 
equal than $k$. This is equal to the number of monomials  
of degree $k$ in $d+1$ variables and by the previous result is
equal to 
\begin{equation}
    \label{eq:degreelessk}
    \binom{k+d}{d} . 
\end{equation}
In particular, \eqref{eq:degreek} shows that the dimension of  $S^ { k-2}( \complexes^d)$  is
\[
\binom{k+d-3}{d-1},
\]
and our task is to show that as a system of linear equations in the entries of $A$, the  
number of independent condtions in \eqref{eq:condition} yields exactly 
\begin{equation}
\label{e:codim}
   \dim ( S^k ( \complexes ^d) ) - \dim (S^{k-2} ( \complexes^d))
   = \binom{k+d-1}{d-1}-\binom{k+d-3}{d-1}. 
\end{equation}
We start by noticing that for $\{e_j\}_{j=1}^d$ the standard orthonormal basis of $\reals^d$ and 
\[
    \zeta=ie_d+\sum_{j=1}^{d-1} z_j e_j, \qquad \sum_{j=1}^{d-1} z_j^2=1
\]
it holds that $\zeta \cdot \zeta=0$. Therefore (a multiple of) 
$\zeta $ lies in  $ \mathcal{V}$ and \eqref{eq:condition} yields  that
\[
A^{(k)} \cdot  \zeta^k=0.
\]
We will further denote $z_{d-1}=\omega$ and observe that without loss of generality we can take 
$\omega=z_{d-1}=\sqrt{1-\sum_{j=1}^{d-2} z_j^2}$,   with
the vector $z'=(z_1,\dots, z_{d-2}) \in 
\complexes^{d-2}$ satisfying that $|z'|^2=\sum_{j=1} ^ { d-2} |z_j|^2\le 1$. 
That is,  $z'$ is an 
arbitrary member of the complex $d-2$ ball.

 The following argument is better understood by introducing the multiindices more 
 commonly found in the pde literature.  Thus   
for  $ \alpha,\beta \dots \in \{1, \dots, d\}^k$ we use   
$\bar \alpha, \bar \beta\dots$ 
for the corresponding multiindex in $\{0, \dots, k\}^d$ 
where $ \bar \alpha _j = \#\{ \ell : \alpha _\ell = j\}$. 
Thus $ \bar \alpha _j$ counts the number of times that $\alpha$ takes the value $j$. 
For $ \bar \alpha \in \{0, \dots, k\}^d$, we will write 
$ \bar \alpha = (\tilde \alpha , \bar \alpha _{ d-1}, \bar \alpha _d) $ with 
$ \tilde \alpha \in \{0, \dots, k\} ^ { d-2}$. Since our tensors are symmetric, the value of 
$A_ \alpha$ depends only on $\bar \alpha$, thus we will write $A_{ \bar \alpha} = A_\alpha$. 


Thus with $ \zeta$ as above, we may use  multilinearity to write
\begin{equation}
\label{eq:motherpolinomial} 
A^{(k)} \cdot  \zeta^k=\sum_{|\bar{\alpha}|\le k} c_{\tilde{\alpha},\bar{\alpha}_{d-1},\bar{\alpha}_d} 
A_{\tilde{\alpha},\bar{\alpha}_{d-1},\bar{\alpha}_d}( z' )^{\tilde{\alpha}} \omega^{\bar{\alpha}_{d-1}} i^{\bar{\alpha}_d}
\end{equation}
for suitable  real and 
positive coefficients $c_{\tilde{\alpha},\bar{\alpha}_{d-1},\bar{\alpha}_d}$. 
Notice that if 
$\bar{\alpha}_{d-1}$ is even, 
\begin{equation}\label{eq:omega} \omega^{\bar{\alpha}_{d-1}}
=(1-\sum_{j=1}^{d-2} z_j^2)^{\bar{\alpha}_{d-1}/2}
\end{equation}
is a polynomial in $z'$ and if $\bar{\alpha}_{d-1}$ is odd,
\[\frac{\omega^{\bar{\alpha}_{d-1}}}{\omega} 
=(1-\sum_{j=1}^{d-2} z_j^2)^{(\bar{\alpha}_{d-1}-1)/2} \]
is another  polynomial.
Therefore we can  write
\[ 
    0=A^{(k)} \cdot \zeta^k=P_1(z')+\omega P_2( z')    
\]
where $P_1$ is a polynomial of degree $k$, grouping all the even powers of $\omega$ and $P_2$ 
is a polynomial of degree $k-1$ grouping the odd powers of $\omega$. 

Now if $P_2( z') \neq 0$, then 
\[\sqrt{1-\sum_{j=1}^{d-2} z_j^2}=\omega=\frac{P_1( z')}{P_2( z')} \]
that is the singular function $\sqrt{1-\sum_{j=1}^{d-2} z_j^2}$ would be equal to a 
meromorphic rational function in the entire complex unit ball which is impossible. 
Thus we deduce that both polynomials vanish identically:
\[P_1( z') = P_2( z')=0\]
for $ z'$ in the $d-2$ complex unit ball.
Next, for $\ell=1,2$ we will write 
\[
P_1(z) = \sum _ { |\tilde \alpha | \leq k} c^{1}_ { \tilde \alpha } z^ {\tilde \alpha},\qquad
P_2(z) = \sum _ { |\tilde \alpha | \leq k-1} c^{2}_ { \tilde \alpha } z^ {\tilde \alpha},
\] 
where  for $c_{\tilde{\alpha}}^1$, $\tilde{\alpha}$ belongs to $   \{0, \dots, k\} ^ { d-2}$ and for $c_{\tilde{\alpha}}^2$, $\tilde{\alpha} $ belongs to $\in  \{0, \dots, k-1\} ^ { d-2}$. A polynomial vanishes identically if and only if its coefficients are null.
We infer that necessarily,
\begin{equation}\label{e:system} 
c^{1}_{\bar{\alpha}}=0, \ |\bar \alpha | \leq k, \qquad c^2_{\bar \alpha }=0, \ |\bar \alpha | \leq k -1. 
\end{equation}
Furthermore, each entry in  $c^{\ell}_{\bar{\alpha}}$  can be read as 
a linear  combination of the components of the tensor $A^{(k)}$.  Thus  \eqref{e:system}
is a system of linear equations where the variables are the components  of $A^{(k)}$.

We claim that this is a set of independent equations and will establish this claim below. Assuming the claim we will count the number of equations in this system and use this to prove our result. Recall that $P_1( z')$ 
is a polynomial of degree $k$ in $d-2$ variables and $P_2( z')$ is a polynomial 
of degree $k-1$. Thus \eqref{eq:degreelessk}  yields that the number of equations 
in  \eqref{e:system} is 
\[ 
\binom{k+d-2}{d-2}+\binom{k+d-3}{d-2}
\]
Now, a calculation will verify that 
\[
\binom{k+d-2}{d-2}+\binom{k+d-3}{d-2}= \binom{k+d-1}{d-1}-
\binom{k+d-3}{d-1}. 
\]
This shows that the number of equations in \eqref{e:system} is  the value we found 
in  \eqref{e:codim}. This completes the proof of the theorem, except for the 
verification of our claim that the equations in \eqref{e:system} were independent. 

The coefficients $c^\ell_{\tilde{\alpha}}$ might be rather complicated as one needs to introduce
\eqref{eq:omega} into \eqref{eq:motherpolinomial} and regroup the powers of $z'$.
We deal with $c^1_{\tilde{\alpha}}$.
For each  multiindex $\tilde{\alpha} \in \{0, \dots, k\} ^{ d-2}$, we  define
$\tilde{\alpha}_0=(\tilde{\alpha},0,k-|\tilde{\alpha}|) \in \{0, \dots, k\} ^ { d}$.  We observe that the variable  
$A_{\tilde{\alpha }_0}$  appears in the expression  $c^1_{\tilde{\alpha}}$ but not
in $c^1_{\tilde{\beta}}$ for  $\tilde{\beta} \in \{0, \dots, k\} ^{ d-2}$
with $\tilde{\beta} \neq \tilde{\alpha}$. This is because in \eqref{eq:motherpolinomial}, $A^{(k)}_{\bar{\alpha }}$ 
does not multiply powers of $\omega$. Therefore each equation $c^1_{\tilde{\alpha}}=0$ can be written as 
\[ 
A^{(k)} _{\tilde{\alpha}_0}
=\sum_{\bar{\alpha}\neq \tilde{\beta}_0, \tilde{\beta} \in \{0, \dots, k\} ^{ d-2} } 
c_{\bar{\alpha}} A_{\bar{\alpha}}
\]

That is,   in each of these linear equations we have a new variable which does not appear in the rest. 
In other words,  the set of equations $c^{1}_{\bar{\alpha}} =0 $ are independent equations on the 
coefficients $A_{\bar{\alpha}}$. Exactly, the 
same argument works for the system  $c^{2}_{\bar{\alpha}} =0 $. 
Therefore, we have verified our claim of the independence 
of the system \eqref{e:system} and the proof for $k \ge 2$ concluded. 

For the case $k=1$, we need less machinery. Pick distinct indices  
$j$ and $k$ with $1\leq j,k \leq d$ 
and consider the vectors $\zeta _{ \pm} = e_ j \pm i e_k$. 
Since $A^{(1)}\cdot \zeta _{\pm}=0$, we have $A^{(1)}_j \pm i A^{(1)}_k=0$. 
It follows that each component of $A^{(1)}$ is zero. 
\end{proof}
\subsubsection{Proof of the structure theorem}
Finally we gather all our knowledge to complete the proof of 
Theorem~\ref{thm:tensorstructure}

\begin{proof}[Proof of the structure theorem]
We begin by considering the case when $ \xi = e_d$, the 
unit vector in the direction of the last coordinate. 
We let $ A^{(k)}$ be a symmetric tensor of order $k$ and assume 
that $ A^{(k)}\cdot \zeta ^k=0$ 
for $ \zeta \in {\cal V}_{e_d} = {\cal V} \cap (\complexes ^{d-1} \times \{0\})$. 
We define a tensor $ A' \in S^k ( \complexes ^{d-1})$ by 
$ A' _ \alpha= A^{(k)}_\alpha$ for all $ \alpha \in \{ 1, \dots, d-1\}^k$. In 
other words, we drop all entries of $A^{(k)}$ for which at least one index $\alpha _j = d$. 
We observe that we have $A' \cdot \zeta ^k=0 $ for all $ \zeta $ in the $d-1$-dimensional version of 
$ {\cal V}$ since if $\tilde \zeta = (\zeta, 0)\in {\cal V}_{e_d}$, then   $A' \cdot \zeta^k =A^{(k)} \cdot\tilde \zeta^k$. 

Thus, we may apply Proposition \ref{prop:tensoronV} 
to $A'$ to write $ A' = I_2'\otimes B'$ for some $ B' $ in
$S^ { k-2}(\complexes ^{d-1})$. (We use $I_2' $ for the identity $d-1 \times d-1$ matrix.) 
We now extend $B' $ to be a tensor in $S^k ( \complexes ^d)$ by defining 
$$
B_ \alpha = \begin{cases} B'_ \alpha, \qquad & \alpha \in \{1, \dots, d-1\}^k\\
                                 0, \qquad & \alpha_j = d \ \mbox{for some $j$}
                                 \end{cases}
$$
Next we consider  the tensor $ \tilde C= A^{(k)} - I_2 \otimes B$.  By our construction of 
$ B'$ and $B$, we have
$ \tilde C_\alpha=0$  for all indices $\alpha$ in $\{ 1, \dots, d-1\}^k$.  
 We define a tensor 
$C\in S^ { k-1}( \complexes ^d)$, by  $ C_ \alpha = \frac  k {1+n(\alpha)} \tilde C_{d \alpha }$ where
$n(\alpha ) = \# \{ j : \alpha_j = d\}$.  We claim that we have 
$ \tilde C= e_d \otimes  C$. Given the claim, 
the desired decomposition  $ A^ {(k)} = I_2 \otimes B + e_d \otimes C$  holds. 

To establish the claim, we begin by observing that for $ \alpha \in \{ 1, \dots, d-1\} ^k$, we have 
$\tilde C_ \alpha = (e_d \otimes C )_\alpha =0.$  If $ \alpha$ is a multiindex with at least 
one entry $ \alpha_j =d$, then we may form a multiindex $ \alpha'$ of length $k-1$ 
by omitting the first index $ \alpha_j $ for
which $ \alpha _j=d$. Recalling our definition of the symmetric tensor product
\eqref{eq:deftensorproduct}, 
we have
$$
( e_d \otimes C)_\alpha = \frac 1 { k!} (\# \{ \pi \in \Pi(k) : \alpha _{ \pi (1) } =d \} )
\frac k {n( \alpha ' )+1} \tilde C_{ d \alpha'} .
$$ 
Since $ \# \{ \pi : \alpha_{\pi (1)} =d\} 
= (n(\alpha') +1) (k-1)!$, we have that 
$( e_d \otimes C)_\alpha=\tilde C_{d\alpha'}$.
Since the multiindex $\alpha$ is a permutation
of $d\alpha'$ and the tensor $ \tilde C$ is symmetric, 
it follows that $ \tilde C_\alpha= \tilde C_{d\alpha'}$ which establishes the claim. 

Before removing our restriction that $ \xi =e_d$, we discuss how a tensor transforms under a linear transformation. 
If $ R$ is a $d\times d $ matrix and   $A^{(k)}\in S^{k}(\complexes^d)$,   we may define 
a tensor $ RA^{(k)}$ by  
$$
 (RA^{(k)})_\alpha = \sum  _{\beta \in \{1,\dots,d\}^k}  R_{ \alpha_1\beta_1}\dots R_{ \alpha_k\beta_k}
 A^{(k)}_{ \beta_1\dots\beta_k} .
$$
We observe that we have that $ (RA^{(k)})\cdot B^{(k)}= A^{(k)}\cdot (R^tB^{(k)})$ where
$R^t$ is the transpose of $R$ and that we have $ R(A\otimes B)= (RA)\otimes (RB)$.

We are ready to consider the case when $ \xi $ is a general vector 
in $ \reals ^d$. In the case that $ \xi =0$, 
then $ {\cal V}_\xi =\cal V $ and the desired result follows directly from 
Proposition \ref{prop:tensoronV}. 
If $ \xi \neq 0$, we choose an orthogonal matrix $R$   which satisfies
$R^t e_d = \xi/|\xi|$. Since $R$ is orthogonal, we have $R \zeta \in \cal V$ for all $\zeta\in{\cal V}$. 
Also, $R\zeta\cdot e_d = \zeta \cdot R^t e_d$, thus $R({\cal V}_\xi)={\cal V} _{e_d}$. 
Given $ A ^ {(k)} \in S^k ( \complexes ^d)$ we 
have $ (RA^{(k)}) \cdot (R\zeta) ^k = A^{(k)} \cdot  \zeta ^k$ and thus $ RA^{(k)}\cdot \zeta ^k=0$ for
$ \zeta \in {\cal V }_{e_d}$. 
 
We apply the  case where $ \xi = e_d$ to find $ B \in S^{k-2} ( \complexes ^d)$ and 
$C \in S^ { k-1}( \complexes^d)$ 
so that  $R A^{(k)}= I_2 \otimes B + e_d \otimes C$. Applying the transformation $R$, we have 
$$
A^ {(k)}= R^tR A^{(k)}
= R^t ( I_2 \otimes B +e_d \otimes C)
$$
We can rewrite the terms in this last expression as  
\begin{align}
\label{e:rot1}R^t (I_2 \otimes B)  &= (I_2 \otimes R^tB)\\
\label{e:rot2} 
R^t(e_d \otimes C)  &= (R^te_d \otimes R^tC)
 \end{align}
Using the identities (\ref{e:rot1},\ref{e:rot2}) we have 
$$
A^{(k)} = I_2 \otimes R^tB + (\xi/|\xi| ) \otimes R^tC=
I_2 \otimes R^tB + \xi \otimes (|\xi|^{-1}R^tC).$$
This completes the proof of the Theorem. 
\end{proof}

\subsection{Conditions on directional derivatives on the Fourier side}
\label{s:derivcond}
Our next step is to show that if we have \eqref{eq:conditionmaster} (for the same set of $\zeta$, 
$\xi$ and $\omega$ as in Theorem~\ref{thm:uniquenessmaster}),
then we 
also have that each term in these sums also vanishes. 

\begin{lemma}
\label{l:pieces}
Let $\{A^{(k)}\}_{k=0}^{k_0}$ be a  
collection of compactly supported $S^k(\complexes^d)$-valued distributions 
satisfying the  main equation \eqref{eq:conditionmaster} 
with 
$r_1+ r_2 \leq R$. 
Then 
\begin{equation}
\label{e:pieces}
\langle A^{(k)} \cdot (\zeta^{j} \otimes \omega^{k-j}),  
(\omega \cdot x)^r e^{-i x \cdot \xi} \rangle=0
\end{equation}
for $j=0, \dots, k$ and  $r\geq 0$ such that $r+k-j \le R$.  

Equivalently, 
we obtain the following conditions on the Fourier transform side 
with $ j=0, \dots, k$, $ r \geq 0$ and $r+j \leq R$: 
\begin{equation}
\label{e:derivativesoftensors}
(\omega \cdot \nabla_\xi)^r \hat A^{(k)}(\xi)\cdot (\zeta^{k-j}  
\otimes \omega^{j} )=0. 
\end{equation}
\end{lemma}

\begin{proof}
We fix $j$ in the main equation \eqref{eq:conditionmaster} and vary the powers in the amplitudes 
$a_1$ and $a_2$ to obtain the result. We relabel the main equation by replacing $k$ by $\ell -j$ and
obtain 
that for $j =0, \dots, k_0 $ we have that
\begin{equation} 
\label{eq:remainequation}
\sum_{\ell = j}^{k_0} \binom \ell  j \langle A^{(\ell )} 
\cdot (\zeta^{j} \otimes D^{\ell -j})
a_1, a_2 \rangle 
=0
\end{equation}
We will obtain \eqref{e:pieces} by induction on $k$ from $j$ to
$k_0$. 

To set up the proof, 
will fix $k$ and $j$  and consider   family of equations indexed by 
$r\geq 0 $ and $\ell $ 
\begin{equation}
\label{eq:piecesinduction}
\langle A^{(\ell )} \cdot (\zeta^{j} \otimes 
\omega^{\ell -j}), (\omega \cdot x)^r e^{-i x \cdot \xi} \rangle=0, \quad j \leq 
\ell \leq k , \ r+ \ell -j \leq R . 
\end{equation}
 Our inductive claim is that
if \eqref{eq:piecesinduction} holds for the family determined 
by $k$ and $j$ for some $k$ with $k=0, \dots, k_0-1$, then we 
also have \eqref{eq:piecesinduction}
for the family determined by $k+1$ and the same $j$. 

We consider the base case with $k=j$ (and, of course, $k$ is one of $0, \dots, k_0$). 
We choose  $a_1=1, a_2=(\omega \cdot x)^r$ with $r \le R$. With this
choice, it holds that $D^{\ell -j} a_1=0$ as long as $\ell >j$.  From this,  the 
only remaining   term in \eqref{eq:remainequation}
is 
\[ 
\langle A^{(j)} \cdot \zeta^j,(\omega \cdot x)^r e^{-ix\cdot \xi} \rangle=0 
\]
as claimed.

Now, we establish the inductive claim. Thus we assume that 
\eqref{eq:piecesinduction} holds for
the family determined by $k$ and $j$ and we need to 
establish \eqref{eq:piecesinduction} for $\ell  = k+1$.  
Thus we need to consider exponents $k+1-j $ and $r$ with $ r+k+1-j \leq R$. 
We  fix 
the amplitudes in \eqref{eq:remainequation} for the rest of the argument as 
 $ a_1 (x) = ( \omega \cdot x ) ^{k+1-j}$ and  
$a_2(x) = ( \omega \cdot x ) ^{r} e^ { -ix\cdot \xi } $.
Note that the corresponding powers $r_1,r_2$ in \eqref{eq:remainequation}  satisfy 
$r_1+r_2 \le R$.

We next compute the  derivative of order $\ell -j$ of $a_1$ to obtain
\begin{equation}\label{eq:derivativeweight}
    D^{\ell -j} a_1(x) = \begin{cases}
 \displaystyle \frac {(k+1-j)!}{(k+1-{\ell })!} i^{\ell -j} 
 (\omega\cdot x)^{k+1-\ell}\omega^{\ell -j} ,
   \quad& \ell\leq k+1 \\
    0, &  \ell  >k+1. 
    \end{cases}
    \end{equation}
Therefore,
\[
\langle A^{(\ell)} \cdot (\zeta^{j} \otimes D^{\ell -j}) a_1, a_2 \rangle =
\frac {(k+1-j)!}{(k+1-\ell )!} i^{\ell-j}
\langle A^{(\ell )} \cdot (\zeta^{j} \otimes 
\omega^{\ell -j}), (\omega \cdot x)^{r+k+1-\ell } 
e^{-i x \cdot \xi} \rangle.
\]

We may use the induction hypothesis that  \eqref{eq:piecesinduction} 
holds on the family determined by $k$ and $j$ to conclude that,
\[\langle A^{(\ell )} \cdot (\zeta^{j} \otimes D^{\ell -j} )a_1, a_2 \rangle=0, \quad \ell = j, \dots, k.  \]
Using the second line of \eqref{eq:derivativeweight} we see 
 the terms with $\ell >k+1$ 
vanish as well. Thus (for these amplitudes and for $j$) the equation \eqref{eq:remainequation} reduces to 
\[ 
\langle A^{(k+1)} \cdot (\zeta^{j} \otimes \omega^{{k+1}-j}),  
(\omega \cdot x)^r e^{-i x \cdot \xi} \rangle=0.
\]
Thus we have shown that \eqref{eq:piecesinduction} holds 
for $\ell =  k+1$.
This completes the proof of the 
inductive step and hence the proof of \eqref{e:pieces}. The second version of
our result 
\eqref{e:derivativesoftensors} follows from \eqref{e:pieces} using 
standard properties of the Fourier transform and interchanging the roles of $j$ and $k-j$. 
\end{proof}

\subsection{Proof of Theorem~\ref{thm:uniquenessmaster}}

We now give the proof of Theorem \ref{thm:uniquenessmaster}. The heart of the proof 
is  the following lemma, Lemma \ref{l:full}. Using 
our tensor structure theorem, Theorem \ref{thm:tensorstructure}, we are able to 
establish that 
a tensor-valued distribution $A^{(k)}$ satisfying \eqref{e:derivativesoftensors} must 
have a special structure. By iterating this argument, we impose additional 
restrictions until we are finally able to  conclude that $ A^{(k)}$ is zero.  As 
a byproduct of this argument we see that in the case we do not have 
\eqref{e:derivativesoftensors} for enough values of $r$ and $j$ to conclude that
$A^{(k)}$ is zero, we still obtain some information about the structure of $A^{(k)}$.

Throughout this section,  $A^{(k)}$ represents a compactly supported 
tensor valued distribution and thus its Fourier transform, $ \hat{A}^{(k)}$, will 
be  a 
tensor valued $C^\infty$-function. We will often  write   $ \hat A ^ {(k)}$ 
in terms of tensor valued $C^ \infty$ functions $B^{(l)}$ of order $l<k$. 
The precise functions $B^{(l)}$ are less 
important and we will allow $B^{(l)}$ to denote different functions as we 
proceed through the proof.

\begin{lemma}
\label{l:full}
    Assume that  $A^{(k)}$ satisfies  \eqref{e:derivativesoftensors} for nonnegative integers $r$ 
    and $j$ with  $r + j \leq k$. Then $ A^{(k)}=0$. 
\end{lemma}

\begin{proof}
    The heart of the proof of the Lemma
    is two inductive claims. We state and prove these 
    claims and then give
    the details needed to complete 
    the proof. Throughout this argument we will use that 
    $ \omega =  \re \zeta$ so 
    that $ \omega \cdot \zeta =1$ and $ \xi \cdot \zeta = \xi \cdot \omega =0$. 

    Our first claim is that if for some $n$ with $n=0, \dots, \lfloor k/2\rfloor -1 $ and 
    $s$ with $s=0, \dots, k-2n$, we have \eqref{e:ind1} below, then we also have \eqref{e:ind1} with $s$ replaced by $s+1$ and the same value of $n$
    \begin{equation} 
    \label{e:ind1}
        \hat A^{(k)}( \xi) = I_2 ^ { n+1}\otimes B^ { (k-2n-2)}(\xi) + I^n_2 
        \otimes \xi ^s \otimes B^{(k-2n-s)}(\xi). 
    \end{equation}

    To establish this first claim, we first  observe that $ n+s \leq 
    2n+s\leq k$ and thus we are able to use \eqref{e:derivativesoftensors} for $r=s$ 
    and $ j=n$. We  substitute our induction hypothesis into \eqref{e:derivativesoftensors} and obtain
    \begin{multline*}
 0= (\omega\cdot \nabla _\xi)^s  \hat A^{(k)}( \xi) \cdot
 ( \zeta ^{k-n}\otimes \omega ^n )
 = (I_2 ^ { n+1}\otimes (\omega\cdot \nabla _\xi)^sB^ { (k-2n-2)}(\xi) )
 \cdot ( \zeta ^{k-n}\otimes \omega ^n ) \\
  + (\omega\cdot \nabla _\xi)^s( I^n_2 
        \otimes \xi ^s \otimes B^{(k-2n-s)}(\xi))\cdot ( \zeta ^{k-n}
        \otimes \omega ^n ).
    \end{multline*}
    The first term on the right will vanish thanks to \eqref{e:id2}. 
    We expand the derivative in the second term using the product rule and
    observe that every term with a positive power of $\xi$ will
vanish since $\zeta \cdot \xi = \omega \cdot \xi =0$. Using \eqref{e:id1} 
and \eqref{e:id3} the previous
displayed equation simplifies to 
 $$
    0=(\omega \cdot \nabla_\xi) ^s \hat A^{(k)}(\xi) \cdot (\zeta^{k-n}\otimes\omega^n)
    = c s! (\omega \cdot \zeta) ^ { n+s} B^ {(k-2n-s)}(\xi) \cdot \zeta ^ { k-2n-s}. 
    $$
    We conclude $ B^ {(k-2n-s)}\cdot \zeta ^ { k-2n-s} =0$. Using our tensor structure theorem, 
    Theorem \ref{thm:tensorstructure}, we 
    can write $ B^ { (k-2n-s)}( \xi) = I_2 \otimes B^ {(k-2n-s-2)}(\xi)+ \xi 
    \otimes B^{(k-2n-s-1)}(\xi) $. 
    Substituting this expression into \eqref{e:ind1}  and regrouping gives \eqref{e:ind1} with $s$ 
    replaced by $s+1$. Thus we have established our first claim. 

    The second inductive claim is that if for some $n$ with $ n=0, \dots, \lfloor k/2\rfloor $, 
    we have
    \begin{equation}
        \label{e:ind2} 
        \hat A^{(k)}(\xi) = I_2 ^n \otimes B^ {(k-2n)}(\xi),
    \end{equation}
    then we obtain the same result with $ n$ replaced by $n+1$. 
    
To establish the second claim, we use \eqref{e:derivativesoftensors} with $ r=0$ and $j=n$. 
Using \eqref{e:id3}, we conclude
$$ 0= c ( \omega \cdot \zeta) ^ n B ^ {(k-2n)} \cdot \zeta ^ {k-2n}. 
$$
Since $ B ^ {(k-2n)} \cdot \zeta ^ { k-2n} =0$, we may use our tensor structure theorem,
Theorem \ref{thm:tensorstructure},  to conclude that
\begin{equation} 
\label{e:IndDecomp}
B^ {(k-2n)}(\xi) = I_2 \otimes B^{(k-2n-2)}(\xi) + \xi \otimes B^{(k-2n-1)}(\xi).
\end{equation}
Substituting this into \eqref{e:ind2} 
gives \eqref{e:ind1} with $s=1$ (and the same $n$). Now we apply our inductive claim 1 
for $ s=1, \dots, k-2n$ and conclude 
that 
$$ \hat A ^{(k)} ( \xi) = I_2 ^ { n+1} \otimes B^ { (k-2n-2)}( \xi) 
+ I_2^n\otimes\xi^{ k-2n+1}\otimes B^ { (-1)}( \xi). 
$$
Recalling  our convention that terms involving tensors 
of negative order are in fact zero, we see that 
we have established
\eqref{e:ind2} holds with $n $ replaced by $n+1$. This 
completes the proof of our second claim. 
Note that when $ n= \lfloor k/2\rfloor$, then both terms 
on the right of \eqref{e:IndDecomp} are zero
so that we conclude that $\hat A ^{(k)}=0$.

To complete the proof of the Lemma, we observe that  we may obtain \eqref{e:ind2} 
with $n=0$ by letting  
$ B^ {(k)} = \hat A^{(k)}$. 
Applying our second inductive claim  for $n=0, \dots, \lfloor k/2\rfloor$, we conclude 
that  $\hat A^{(k)}=0$ 
and the Lemma is proved. 
\end{proof}

We observe that if we have \eqref{e:derivativesoftensors} for a smaller 
range of values of $r+j$, we can still follow the steps in the 
proof of Lemma \ref{l:full}. However, the induction will stop 
before we can conclude the $ \hat A^{(k)}$ is zero, but we still obtain some 
information about the structure of $\hat A^{(k)}$. The next 
Lemma gives a precise result in this direction. 
\begin{lemma}
\label{l:partial}
Let 
    $ R+1\leq k \leq 2R+1$ and suppose that $ A^{(k)}$ 
    satisfies    \eqref{e:derivativesoftensors} 
    for $r$ and $j$ with $r+j \leq R$. 
    Then for $ s= 2R+1-k$, we have
    $$
    \hat A ^{(k)} ( \xi) = \xi ^ { s+1} \otimes B ^ { (k-s-1)}(\xi). 
    $$
\end{lemma}

\begin{proof}
    As with the previous Lemma, we give an argument using two inductive claims. 
    Our  first claim, is that if we have 
    \begin{equation} 
    \label{e:pind1}
    \hat A ^ {(k)}(\xi) = \xi ^s \otimes I_2 ^n \otimes B^ { (k-s-2n)} 
    + \xi ^ { s+1} \otimes B^ { (k-s-1)}( \xi) 
    \end{equation}
    with $s+n\leq R$,  then we have \eqref{e:pind1} with $ n$ 
    replaced by $n+1$. To establish the first claim, 
    we apply \eqref{e:derivativesoftensors} with $ r= s$ and $j=n$, 
    use \eqref{e:pind1} and simplify 
    to obtain $(\omega\cdot\zeta)^nB^{(k-s-2n)}(\xi)\cdot\zeta^{k-s-2n}=0$. 
    Now, 
    Theorem \ref{thm:tensorstructure} tells that $ B^ { (k-2n-s)}(\xi) 
    = I_2 \otimes B^{(k-2n-2-s)}( \xi) + \xi \otimes B^ {(k-2n-1-s)}( \xi)$. Substituting 
    this expression into \eqref{e:pind1} and regrouping gives 
    this result with $s$ replaced by $s+1$. 

   Our second claim is that if for some $s=0, \dots , 2R+1-k$ we have 
    \begin{equation}
    \label{e:pind2}
        \hat A^{(k)}( \xi) = \xi ^s \otimes B^{(k-s)}( \xi) 
    \end{equation}
     then \eqref{e:pind2} holds with 
    $s$ replaced by $s+1$ and the same value of $k$. 

    To establish the second claim, we begin by using \eqref{e:derivativesoftensors} 
    with  $ r=s$ and $ j=0$. Note that   since   
    $R<k$, we  have that 
    $s\leq 2R+1-k \leq R$ and  we may use 
    \eqref{e:derivativesoftensors}.  Simplifying we conclude that
    $$
    c s! ( \omega \cdot \zeta )^s B^ { (k-s)}( \xi) \cdot \zeta ^ { k-s} =0.
    $$
    Using our tensor structure result, Theorem \ref{thm:tensorstructure}, 
    we obtain \eqref{e:pind1} with $s$ and $n=1$. From here, we apply 
    our first inductive claim for $ n=1, \dots, R-s$ 
    and obtain 
    $$
    \hat A ^ {(k)}(\xi) = \xi ^s \otimes I _2 ^ { R-s +1} \otimes 
    B^ { (k-s-2R+ 2s -2)} ( \xi) + \xi ^ { s+1} \otimes B^{(k-s-1)}. 
    $$
    Since  $ s\leq 2R+1-k $, the order of the tensor  $B^{(k-s-2R+2s-2)}$ satisfies
    $$
    k+ s -2R-2 \leq k + 2R+1-k-2R-2=-1. 
    $$
    Thus for some $n \leq R-s$, the first term is in fact zero and we obtain 
    \eqref{e:pind2} with $s$ replaced by $s+1$. 
    
    To establish the Lemma, we fix $k$ and observe that \eqref{e:pind2} 
    holds with $s=0$. From here we use the second inductive claim until we reach
    $s=2R+1-k$ and obtain the Lemma. 
\end{proof}

\begin{proof}[Proof of  Theorem~\ref{thm:uniquenessmaster}]
We let $\{A^{(k)} : k=0, \dots, k_0\}$ be a collection 
satisfying \eqref{eq:conditionmaster} with $ \xi, \zeta,\omega,r_1, r_2$ 
as in the Theorem. From here Lemma \ref{l:pieces} gives the condition \eqref{e:derivativesoftensors} with $ r+j \leq R$. Now Lemma \ref{l:full} gives our 
theorem in the case $k \leq R$ and Lemma \ref{l:partial} gives our result for the
case $R+1\leq k \leq 2R+1$. 
\end{proof}

\rv{
Before giving the final details for the proof of our main theorem, we explain our choice that the 
coefficients $A^{(k)}$ lie in   $\tilde W^{k-s,p}(\Omega)$. The coefficients enter into our 
construction of the CGO solutions 
in two ways. First, in the proof of Proposition \ref{p:CGOestimate}, we use that $A^{(k)}$ 
lies in $ \tilde W^ { k-s, p}( \Omega)$  with $p\geq 2$ in order to 
obtain the estimate \eqref{e:CoeffTerms}. 
This is natural since the $X^{-\lambda}$ spaces are based on the $L^2$ norm. We also need 
the condition  \eqref{e:CoeffCond} in order to use Proposition \ref{prop:aprioriestimate} to solve the 
equation $L_\zeta \psi = f$. The condition  $1/p < (m-s)/d$ allows us to 
use Sobolev embedding to obtain \eqref{e:CoeffCond}. We believe that \eqref{e:CoeffCond} 
involves the right order smoothness: the study of the $m$th power 
of the Laplacian requires that the coefficients $A^{(k)}$ lie in a Sobolev space of order
$k-m$ or that $ s= -m$.  Our averaging argument, 
Proposition \ref{p:CGOestimate}, puts a stronger restriction on the value of $s$
for which we can study the inverse problem. We expect 
that improving or replacing this argument will be a fruitful area of investigation. 
}

\begin{proof}[Proof of Theorem \ref{t:Main}]
We have two operators $L_\ell$, $\ell=1,2$ with coefficients as 
in the statement of this Theorem. 
Note that since 
the coefficient
$A_\ell ^{(k)}\in \tilde W^ {k-s , p}( \Omega)$,  then we 
may use a result of Brown and Gauthier  
\cite[Proposition A.3]{MR4455267} to show that coefficients also 
satisfy \eqref{e:CoeffCond}.
  Thus we may use 
Theorem \ref{t:MEHolds} to obtain  \eqref{e:MEHolds}. 

 From here, we may use Lemma \ref{l:pieces} to obtain \eqref{e:derivativesoftensors} 
with $ R = \lfloor m/2\rfloor$ (it is natural to make $R$ an integer). 
Now, Theorem \ref{thm:uniquenessmaster} gives that 
$ A^{(k)} = A_1 ^ {(k)}-A_2 ^{(k)}=0$ 
which gives our Theorem. 
\end{proof}

\subsection{Further questions}
\label{s:further}
We close with a few questions that this work leaves open. In this paper, we have made 
a good deal of progress in recovering lower order terms of operators that are
perturbations of the polyharmonic operator. However, there is much to be done. It seems
likely that the order of the Sobolev spaces is not optimal.  In general, we expect that
one should be able to prove uniqueness when the coefficient $A^{(k)}$ is in a 
Sobolev space of order $k-m$. Beyond this is the question of the optimal exponent $p$ which we
have done little to address.   We note that in our earlier work, two of us establish
uniqueness for operators of the form $ \Delta^2 +A^{(0)}$ when $A^{(0)}$ lies in 
a Sobolev space of order $-s$ for any $s<2$. Thus we come very close to the conjecture
stated above. 

We expect that the some of the restrictions on $k_0$ can be removed. We can construct
CGO solutions for at least some operators with a term of order $m$, $ A^{(m)}\cdot D^mu$. 
However we cannot establish the asymptotic estimates that we would need to establish
uniqueness. We do not have a specific conjecture as to what conditions would
allow us to prove uniqueness, but leave it to the
interested reader to consider this topic. 

\rv{
Another source of potential improvements comes from  comparing our results to what is known for $m=1$. For dimensions 
$d=3,4$ and $m=1$, Haberman \cite{MR3397029} has established
that we may recover a zeroth order coefficient $A^{(0)}$ which lies in the 
Sobolev space $W^{-1,d}( \reals ^d)$. Also see further results of Ham, Kwon and Lee \cite{MR4273826}. 
If $A^{(0)}$ is in  $W^{-1,d}$, then
the map 
$ u \rightarrow A^{(0)}u $ maps $W^ { 1,2}$ to $W^ { -1,2}$. To extend this to the polyharmonic operator
we need to find multiplier operators which map $W^{m,2}( \reals ^d)$ to $W^ { -m,2}(\reals ^d)$. If 
$A^ { (k) } \in W^ { k-m, d/m}( \reals ^d)$, then 
the map $u \rightarrow A^ {(k)} \cdot D^k u $ maps  $W^ { m,2}( \reals ^d)$ to $W^ { -m,2}(\reals ^d)$.  We conjecture that a sharp result for the inverse problem for the 
polyharmonic operator  of order $m$ would allow us to recover $A^ {(k)} $ in  $W^ {k -m, d/m}$ for $k=0,\dots, m-1$. 
To avoid complications, 
we have restricted our attention to $d>2m$ and compactly supported coefficients. 
}

Finally, our formulation of the problem with coefficients in $ A^{(k)} \in\tilde W^{k-s,p}(\Omega)$
is an assumption about the behavior of the coefficients near the boundary. This assumption
is convenient in that it allows us to proceed to the main inverse problem without
worrying about some sort of boundary identifiability result. It would be interesting to
have a better sense of the sort of assumptions on the coefficients 
near the boundary are appropriate from the
point of view of applications.

\newpage

\appendix 


\renewcommand{\thetheorem}{\thesection.\arabic{theorem}}
\renewcommand{\theequation}{\thesection.\arabic{equation}}

\section{Some Tensor Algebra}
\label{s:tensorcomp}
In this section, we  collect several computational results related to symmetric tensors.   
These results may be well known
to others, but we needed to develop these results for the 
computations in section \ref{s:uniq}. 
We will use  $\alpha',\alpha''$ to denote various multiindices and 
recall that as in
\eqref{e:ConDef}, 
$\alpha'\alpha''$ is the multindex obtained by concatenation. 

\begin{lemma}
\label{l:ConLem}
Let $j \leq k$ and suppose that $A\in S^{k}(\complexes ^d)$ is a symmetric 
$k$-tensor and $B$ is  an array (which may not be symmetric) 
indexed by multiindices of length $j$, then we have that 
$$ 
A \cdot B = A \cdot \sigma(B).
$$
\end{lemma}
\begin{proof}
    For $ \alpha$ a multiindex of length $k-j$ we have 
    \[
    \begin{split}
     (A \cdot \sigma(B) ) _ \alpha  &=\frac 1 { j!}\sum _{ |\beta| = j} 
     \sum _ { \pi \in \Pi(j)} A _ { \alpha\beta}  B_{ \pi (\beta)} \\
        & = \frac 1 { j!}\sum _{ |\beta| = j} 
        \sum _ { \pi \in \Pi(j)}A _ { \alpha \pi^{-1}(\beta)}  
        B_{ \beta}\\
        & = \frac 1 { j!}\sum _{ |\beta| = j} 
        \sum _ { \pi \in \Pi(j)}A _ { \alpha\beta}  B_{ \beta}\\
        & = \sum _{ |\beta| = j} A _ { \alpha\beta}  B_{ \beta} = (A\cdot B)_\alpha. 
    \end{split}
   \]
    The third equality follows since the tensor $A$ is symmetric and thus we 
    have $ A_{\alpha\pi^{-1}(\beta)}= A_{\alpha\beta}$. 
\end{proof}

\begin{lemma} 
\label{lem:Triplecontraction}
Let $A \in S^k(\complexes^d), B \in S^s(\complexes^d), C\in S^{k-s}(\complexes^d)$. Then
\[
(B \otimes C )\cdot A=\sum _{|\alpha'|=s, |\alpha''|
= k-s} B_{\alpha'} C_{\alpha''} A_{\alpha'\alpha''} 
\]
\end{lemma}
\begin{proof}
Given a multiindex $ \alpha$ with $|\alpha|=k$, we may write 
$ \alpha = \alpha'\alpha''$ where 
$ \alpha'$ is the first $s$ entries in $ \alpha$. Then the components of 
the  tensor product 
$B\otimes C$ may be obtained by symmetrizing the array 
$ (B_{\alpha'}C_{\alpha''})_{\alpha}$.  
With this observation, the Lemma follows from Lemma \ref{l:ConLem}. 
\end{proof}


From a $k$-tuple of vectors $\eta=(\eta_1,\eta_2, \dots, \eta_k) \in (\complexes ^d)^k$
we can form $\eta_1 \otimes \dots \otimes \eta_k \in S^k ( \complexes^d)$ 
but also tensor products of order $s$
by consider subsets of $s$ vectors. We will use the letter 
$J=\{j_1,j_2, \dots, j_s\}$ 
for the set of indices of the selected vectors  and the  
complement of $J$ in $\{1, \dots, k\}$
will be  $J^c$. We denote by $\# J$, the cardinality of $J$. 
Then
$\eta^J=\eta_{j_1} \otimes \eta_{j_2}\otimes \dots \otimes \eta_{j_s}$. 
For example, if 
$k=7$, $J= \{2,3,6\}, J^c=\{1,4,5,7\} $, then 
$ \eta^J = \eta_2\otimes \eta _3 \otimes \eta _6$ and $ \eta ^ { J^c} 
= \eta _1 \otimes \eta _4\otimes \eta _5 \otimes \eta _7$. 
This notation will be used in the  
following Lemma gives an expression for the components of the tensor 
product $ \eta _1\otimes \dots \otimes \eta _k$.

\begin{lemma} \label{l:tenprod}
Let $\eta \in (\complexes^d)^k$
and $\alpha',\alpha''$, fixed multiindices of length  $s$ and  $k-s$.   Then
\[ 
\sum_{\#J=s} (\eta^J)_{\alpha'} (\eta^{J^c})_{\alpha''}
=\binom{k}{s}(\eta_1\otimes \dots \otimes \eta _k)_{\alpha'\alpha''}  
\]
\end{lemma}

\begin{proof}
    We begin with an observation about the group $ \Pi(k)$. For each subset 
    $J \subset \{1, \dots, k\}$ with $ \#J=s$, we choose a permutation 
    $ \pi_J$ with the property that 
    $ \pi_J(\{1, \dots, s\})= J$.  Once these permutations have 
    been chosen, then observe that we can write 
    \begin{equation} 
    \label{e:Pkdecomp}
        \Pi(k) = \cup _J ( \Pi(s) \times \Pi (k-s))\circ \pi _J
    \end{equation}
    and this is a disjoint union. Note that our choice of $ \pi _J$ 
    is not unique, but  the   set of permutations 
    $ ( \Pi(s) \times \Pi(k-s)) \circ \pi _J$ depends only on $J$ 
    and not our choice of $\pi _J$. 

    Next fix a set $J = \{ j_1, \dots, j_s\}$ and we index the 
    complement $J^c=\{\ell_1, \dots, \ell_{ k-s}\}$. For a 
    multiindex $\alpha= \alpha' \alpha''$ with $|\alpha'|=s$ 
    and $|\alpha''|= k-s$, we can write
    \begin{multline}
    \label{e:expj}
    \eta^J_ {\alpha'} \eta^{J^c}_{\alpha''}= \frac 1 { s! (k-s)!}
    (\sum_{ \pi \in \Pi(s) } 
    \eta_{j_{\pi(1)}, \alpha'_1}\dots \eta_{j_{\pi(s)}, \alpha'_{s}})
    ( \sum_{ \pi \in \Pi(k-s) } \eta_{\ell_{\pi(1)}, \alpha''_1}
    \dots \eta_{\ell_{\pi(k-s)}, \alpha''_{k-s}})\\
    = \frac 1 {s!(k-s)!}\sum _{\pi \in ( \Pi(s) \times \Pi(k-s))
    \circ { \pi _J}} \eta _{ \pi(1), \alpha_1 }\dots \eta _ { \pi (k), \alpha_k}
    \end{multline}
    If we sum \eqref{e:expj} over all subsets $J \subset \{1,\dots,s\}$ with $\#J=s$ use the decomposition 
    of $ \Pi (k)$ in \eqref{e:Pkdecomp}, we obtain that 
    $$
    \sum _ {\{ J : \#J =s\}} \eta ^ J _ { \alpha'}\eta^{ J^c}_{\alpha ''} 
    = \frac 1{s!(k-s)!}\sum _{\pi \in \Pi(k)}\eta_{\pi(1), \alpha _1}\dots \eta _{ \pi(k), \alpha_k} 
    = \binom k s ( \eta_1\otimes \dots \otimes \eta _k)_ \alpha. 
    $$
\end{proof}

\begin{lemma}
\label{l:tproduct} 
If $B \in S^{s}(\complexes^d), C \in S^{k-s}(\complexes^d)$ and  
$\eta \in (\complexes^d)^k$. For every $1 \le s \le k$
\[
\binom{ k}{s}(B \otimes C) \cdot ( \eta _1 \otimes \dots \otimes \eta_k )
=\sum_{\# J=s}B \cdot  \eta ^J  \,  C \cdot \eta ^  {J^c}
\]
\end{lemma}

\begin{proof}
We start by writing the definition of 
\[
 B \cdot \eta^J=\sum_{\alpha'} B_{\alpha'} (\eta^J)_{\alpha'}
\]
and thus 
\[(B \cdot \eta^J) (C \cdot \eta^{J^c})=(\sum_{\alpha'} B_{\alpha'} (\eta^J)_{\alpha'})
(\sum_{\alpha''} C_{\alpha''} (\eta^{J^c})_{\alpha''} )
=\sum_{\alpha',\alpha''} B_{\alpha'}
C_{\alpha''} (\eta^J)_{\alpha'} (\eta^{J^c})_{\alpha''}
\]
Summing on $J$, we have
\[
\sum_{\# J=s} (B \cdot \eta^J) (C \cdot \eta^{J^c})=\sum_{\# J=s} 
\sum_{|\alpha'|=s,|\alpha''|=k-s } 
B_{\alpha'} C_{\alpha''} (\eta^J)_{\alpha'} (\eta^{J^c})_{\alpha''}
\]
which equals to 
\[
\sum_{|\alpha'|=s,|\alpha''|=k-s } 
B_{\alpha'} C_{\alpha''} \sum_{\#J=s} (\eta^J)_{\alpha'} (\eta^{J^c})_{\alpha''} 
\]
and by Lemma~{\ref{l:tenprod}} this equals to 
\[
   \binom{k}{s} \sum_{\alpha',\alpha''} B_{\alpha'} C_{\alpha''} 
   (\eta_1\otimes\dots\otimes \eta_k)_{\alpha' \alpha''}
\]
and by Lemma~\ref{lem:Triplecontraction} with $A=\eta_1\otimes \dots \eta_k$ the
last expression equal to 
\[\binom{k}{s} B \otimes C \cdot (\eta_1\otimes \dots \otimes \eta_k) \]
as claimed.
\end{proof}

\begin{lemma}
Suppose $ \zeta \in {\cal V}$, $ \omega \in \complexes ^d$, 
$A^{(k)}\in S^k(\complexes^d)$ and $B^{(j)}\in S^j(\complexes^d)$
    \begin{align}
    \label{e:id1}
        (A^{(k)}\otimes B^{(j)}) \cdot \zeta ^{k+j} 
        &= (A\cdot \zeta ^k)(B\cdot\zeta^j)  \\
        \label{e:id2}
       (A^{(k)} \otimes I_2^n) \cdot ( \zeta ^{k+2n-j} \otimes \omega ^j) &=0 , \qquad j<  n   \\
       \label{e:id3}
           (A^{(k)} \otimes I_2^n) \cdot ( \zeta ^{k+n} 
           \otimes \omega ^n) &=c (\omega\cdot \zeta)^n A^{(k)}\cdot \zeta ^k
    \end{align}
    In the third result, $c=c(k,n)$ is a positive, nonzero constant. 
\end{lemma}  

\begin{proof}
The   tensor $A^{(k)}\otimes B^{(j)}$ is obtained as the symmetrization of
the array $ C_{\alpha\beta}= A^{(k)}_\alpha B^{(j)} _\beta$ for 
multiindices $\alpha,\beta$ 
with $|\alpha|=k$
and $|\beta|=j$. We also have $ \zeta^{k+j} _{ \alpha\beta}
= \zeta_{\alpha_1}\dots \zeta_{\alpha_k}\zeta_{\beta_1}\dots \zeta_{\beta_j}
= \zeta^k_\alpha \zeta^j_\beta$. Using Lemma \ref{l:ConLem}, we have
$$ 
(A^{(k)}\otimes B^{(j)}) \cdot \zeta^{ k+j}= \sum _ { |\alpha|=k, |\beta|=j}
A^{(k)}_\alpha \zeta _{\alpha}^k B_\beta^{(j)}\zeta ^j_\beta 
= (A^{(k)}\cdot \zeta^k )(B^{(j)}\cdot \zeta ^j). 
$$
This gives the first identity, \eqref{e:id1}. 

For the second result, we define $ (\eta_1, \dots, \eta_{ k+2n})$ by 
$ \eta_\ell = \zeta$ for $\ell = 1, \dots, k+2n-j$ and $ \eta _ \ell = \omega$ 
for $ \ell = k+2n-j+1, \dots, k+2n$.  Using Lemma \ref{l:ConLem}, we have that 
\begin{multline*}
(A^{(k)} \otimes I_2^n)\cdot ( \zeta^{k+2n-j}\otimes \omega^{ j})
= \frac 1 {(k+2n)!} \sum _ { \pi\in\Pi(k+2n)} (A^ {(k)}  \cdot ( \eta_{\pi(1)} 
\otimes\dots \otimes \eta_{\pi(k)}) \\
\times \eta_{\pi(k+1)}
\cdot \eta_{\pi(k+2)} \dots \eta_{ \pi ( k+2n-1)}\cdot \eta_{ \pi (k+2n)} . 
\end{multline*}
Since there are only $j$ copies of $ \omega$ and $j<n$, for each 
permutation $\pi$ at least one of the dot products 
$ \eta _{\pi(k+2\ell-1)} \cdot \eta _{\pi(k+ 2\ell)}$  will 
be  $ \zeta \cdot \zeta=0$. Thus each term in the above sum is zero 
and \eqref{e:id2} is proved. 

To establish the third identity \eqref{e:id3}, we let 
$ \eta_1, \dots, \eta_{k+2n}$ be given by 
$$
\eta_\ell = \begin{cases} \zeta, \qquad & \ell= 1, \dots, k+n\\
\omega , \qquad & \ell = k+n+1,\dots , k + 2n.
\end{cases}
$$
For a set of indices $ J= \{ j_1, \dots, j_{2n} \} 
\subset \{ 1, \dots, k+2n\}$, we have
$$
I _2^n \cdot \eta ^J= \frac 1 { (2n)!} 
\sum _{ \pi \in \Pi(2n)} \eta_{j_{\pi(1)}}\cdot \eta_{j_{\pi(2)}} \dots 
\eta_{j_{\pi(2n-1)}}\cdot \eta _ { j _ {\pi (2n)}}
$$
Unless we have exactly one copy of $ \omega$ in each dot product, then the 
resulting summand will vanish and in the circumstance where each dot product 
is $ \omega \cdot \zeta$ or $\zeta \cdot\omega$, the summand is 
$ (\omega\cdot \zeta)^n$. Using Lemma \ref{l:tproduct}, we conclude that 
$$
\binom{k+2n}{n} ( A^{(k)}\otimes I_2 ^n ) \cdot (\zeta ^ { k+n} \otimes \omega ^n)
= \binom{k+n}n \frac { 2^n n!}{(2n)!} A^ { (k)}\cdot \zeta ^k ( \omega\cdot \zeta)^n. 
$$
The combinatorial expression on the right represents the number of subsets 
$J$ with $ \#J=2n$ which contain the indices $\{k+n+1, \dots k+2n\}$ for which  
$ \eta _\ell = \omega$ multiplied by the number of permutations of the set $J$ 
which result in all dot products being equal to $\zeta \cdot \omega$ $(2n)!$ 
from averaging over the permutation group.   The result \eqref{e:id3} follows. 
\end{proof}

\newpage


\bibliographystyle{plain}



\medskip
\small 17 April 2026
\end{document}